\documentclass[letter, 12pt]{amsart}

\def\A{\mathbf{A}}

\title[Archimedean $L$-factors]
{Archimedean zeta integrals on $GL_n \times GL_m$ and $ SO_{2n+1} \times  GL_m$}
\author{Taku Ishii}
\address{Faculty of Science and Technology,
Seikei University, 3-3-1 Kichijoji-Kitamachi,
Musashino, Tokyo, 180-8633, Japan}
\email{ishii@st.seikei.ac.jp}
\author{Eric Stade}
\address{Department of Mathematics, University of Colorado, Boulder, CO 80309, USA}
\email{stade@colorado.edu}
\date{\today}

\usepackage{amsmath} \usepackage{amssymb}
\usepackage{amsopn} \usepackage{amstext}
\usepackage{amsthm} \usepackage{amscd}
\usepackage{amsgen} \usepackage{latexsym}
\usepackage{enumerate}
\usepackage{graphicx}

\newcommand{\newcom}{\newcommand}
\newcom{\renewcom}{\renewcommand}

\def\R{{\mathbf{R}}} 
\def\C{{\mathbf{C}}} 
\def\Z{{\mathbf{Z}}} 
\def\ls{L(s,\pi\times\pi')}
\def\lsnu{L(s,\pi_{\nu}\times\pi' _{\nu})}
\def\zs{Z(s,\varphi,\varphi')}
\def\zsnu{Z_\nu(s,W_{\nu},W'_{\nu})}
\def\ds{\displaystyle}

\def\phi{\varphi}
\def\bs{\backslash}

\theoremstyle{plain}

\newtheorem{thm}{Theorem}[section]
\newtheorem{lem}[thm]{Lemma}
\newtheorem{cor}[thm]{Corollary}
\newtheorem{prop}[thm]{Proposition}

\theoremstyle{definition}

\newtheorem{defn}[thm]{Definition}
\newtheorem{rem}{Remark}

\newcom{\bpf}{{\it Proof. }}
\newcom{\epf}{\hfill $\Box$}

\numberwithin{equation}{section}

\allowdisplaybreaks

\begin{document}
\maketitle
\begin{abstract} In this paper, we evaluate archimedean zeta integrals for automorphic $L$-functions on $GL_n \times GL_{n-1+\ell}$ and on $ SO_{2n+1} \times  GL_{n+\ell}$, for $\ell=-1$, $0$, and $1$. In each of these cases, the zeta integrals in question may be expressed as Mellin transforms of products of class one  Whittaker functions. Here, we obtain explicit expressions for these Mellin transforms in terms of Gamma functions and Barnes integrals.

When $\ell=0$ or $\ell=1$, the archimedean zeta integrals amount to integrals over the full torus. We show that, as has been predicted by Bump for such domains of integration, these zeta integrals are equal to the corresponding local $L$-factors -- which are simple rational combinations of Gamma functions. In the cases of $GL_n \times GL_{n-1}$ and $GL_n \times GL_n$ this has, in large part, been shown previously by the second author of the present work, though the results here are more general in that they do not require the assumption of trivial central characters. (Our techniques here are also quite different. New formulas for $GL(n,\R)$ Whittaker functions, obtained recently by the authors of this work, allow for substantially simplifed computations.)  For the cases of $SO_{2n+1} \times GL_{n}$ and $SO_{2n+1} \times GL_{n+1}$, results such as these have been deduced previously only for a few instances of small rank.

In the case $\ell=-1$, we express our archimedean zeta integrals explicitly in terms of Gamma functions and certain  Barnes-type integrals. These evaluations rely on new recursive formulas, derived herein, for $GL(n,\mathbf{R})$ Whittaker functions. Our results concerning these zeta integrals, in these situations, generalize work of Hoffstein and Murty for the case of $GL_3\times GL_1 $, and parallel similar results, obtained by the first author of the present work and by
Moriyama, for the spinor
$L$-function on $GSp_2 $.

Finally, in our appendix, we indicate an approach to certain unramified calculations, on $SO_{2n+1}\times GL_n$ and $SO_{2n+1}\times GL_{n+1}$, that parallels our method in the corresponding archimedean situation.  While the unramified theory has already been treated using more direct methods, we hope that the connections evoked herein might facilitate future archimedean computations.

\end{abstract}

\section*{Introduction}

In this paper we investigate the behavior, at archimedean places, of zeta integrals associated to certain automorphic forms and $L$-functions, and develop explicit formulas for these integrals at such places.

More specifically: let $ G $ and $G'$ be reductive algebraic groups; let $\pi$ and $\pi'$ be automorphic
cuspidal representations of
$ G (\A)$ and $G'(\A)$ respectively, where
$\A$ denotes the adeles over a global field $F$.
Following Langlands \cite{L}, one may define a global
$L$-function $\ls$, for $s\in\C$ of sufficiently large real part. The Langlands program then predicts that $\ls$ should admit meromorphic continuation in $s$ (with only finitely many poles in $\C$), and a functional equation under $(\pi,\pi',s)\to (\pi^\vee,(\pi')^\vee,1- s)$ (where $\pi^\vee$ and $(\pi')^\vee$ are the representations contragredient to $\pi$ and $\pi'$, respectively).

An especially powerful and fruitful approach to the study of $\ls$ is through the zeta-integral method. At the core of this method is the construction of a certain ``global zeta integral'' $Z(s,\varphi\times\varphi')$, which is, typically, an adelic integral involving automorphic forms associated with $\pi$, $\pi'$, and $s$. The specific form taken by $Z(s,\varphi\times\varphi')$ will depend on the choice of $ G $ and $G'$; see (\ref{zdef}) below for details, in the cases of concern to us in this article. In any case, the automorphicity and analytic properties of the various factors in the integrand of $Z(s,\varphi\times\varphi')$ should bestow upon this global integral both a meromorphic continuation and a
functional equation in $s$.
To deduce similar
properties of $\ls$, one writes both this global $L$-function and the global
integral $Z(s,\varphi\times\varphi')$ as products of local factors, and then compares, at each place of $F$, the
local $L$-function -- which will, generally, amount to a rational combination of gamma functions involving $s$ and the eigenvalues of $\phi$ and $\phi'$ -- with the local zeta integral.

Let us elaborate, in the situations that will concern us  in this article, namely:  the cases $(G,G')=(GL_n,GL_{n-1+\ell})$ and $(G,G')=(SO_{2n+1},GL_{n+\ell})$, for $\ell$ an integer with $|\ell|\le1$.   (Of course, by symmetry, our investigations will also apply to the cases of $(G,G')=(GL_n,GL_{n+j})$, for $j=1$ or $2$.) See, in particular, the works of Jacquet and Shalika \cite{JS1}, \cite{JS2}, and \cite{JS3}, and of Jacquet, Piatetski-Shapiro, and Shalika \cite{JPS}, for those cases where $ G =GL_{n}$; the works of Gelbart, Piatetski-Shapiro, and Rallis \cite{GPR}, of Ginzburg \cite{Gi},  and of Soudry \cite{So1}, \cite{So2} for those cases where $ G =SO_{2n+1}$; and the monograph of Gelbart and Shahidi \cite{GS} for general discussions. We summarize some of the relevant points as follows.

The global zeta integral $\zs$ may be defined  by:
\begin{align}\begin{split}\label{zdef}
&\zs 
\\
&=\begin{cases}
\ds \int_{ GL_n (F)\bs
GL_n (\A)}\phi(g)\,\phi'(g)\,E (g,s)\,dg &\mbox{ if } (G,G') = (GL_n,GL_n) ,  
\\
\ds\int_{ GL_{m} (F)\bs
GL_{m} (\A)}\phi_N\big|_{GL_{m}} (g) \phi' (g)  \,  |\mbox{det}(g)|^{s-1/2}\,dg &\mbox{ if } (G,G') = (GL_n,GL_{m}) \\\noalign{\vskip-24pt}
\\&\qquad\qquad\qquad\mbox{ for }m<n,
\\
\ds \int_{ SO_{2n+1}(F)\bs
SO_{2n+1}(\A)}\,\phi(g)\,E_{\phi'}\big|_{ G } (g,s) \,dg &\mbox{ if } (G,G') = (SO_{2n+1},GL_{n+1})  ,
\\
\ds \int_{ SO_{2m}(F)\bs
SO_{2m}(\A)}\phi_N\big|_{SO_{2m}} (g) \,E_ {\phi'}(g,s)\, dg &\mbox{ if }(G,G') = (SO_{2n+1},GL_{m})\\\noalign{\vskip-20pt}
\\&\qquad\qquad\qquad\mbox{ for }m\le n.
\end{cases}\end{split}
\end{align}
We explain: in each case,
$\phi$ and $\phi'$ are cusp forms in the spaces of $\pi$ and $\pi'$ respectively. Also, $E(g,s)$ denotes a maximal parabolic Eisenstein series on $GL_n$, while $E_{\phi'}(g, s)$ denotes an Eisenstein series on $SO_{2m}$,   induced from the representation $ \pi\otimes |\hbox{det}|^{s-1/2}$ on the Levi component $M_{m}\cong GL_{m}$ of the Siegel parabolic subgroup $P_{m}$ of $SO_{2m}$. Further, in general, $f|_D$ denotes restriction of the function $f$ to the subdomain $D$. (In the cases above, the indicated subdomains must be realized as subgroups of the original domains in suitable ways.)  Finally, $\phi_N$ denotes an averaging  of $\phi $, weighted by  a nondegenerate character $\psi$, over a certain subgroup $N$ of $GL_n$ (in the case of $GL_n\times GL_m$) or of $SO_{2n+1}$  (in the case of $SO_{2n+1}\times GL_m$). (In the latter case, $N$ is trivial when $m=n$.)

Now let us write $\pi =\otimes_\nu \pi_{ \nu}$, with $\pi_{ \nu}$ a
representation of $G (F_\nu)$ for each   place $\nu$ of $F$, and similarly write $\pi'=\otimes_\nu \pi'_{ \nu}$.
Then we have the factorization $\ls=\prod_\nu \lsnu$. We will also assume, from now on, that $\pi'$ is {\it generic}, meaning it admits a nonzero Whittaker model. (This is automatic in the cases $ G' =GL_n,\,GL_{n+1},$ and $GL_{n+2}$ -- automorphic representations of $GL_m$ are always generic. For this same reason, $\pi$ will, in our situation, always admit a nonzero Whittaker model.)
Then, by
the Rankin-Selberg unfolding method, the global integral $\zs$ may itself be expressed as a product of local factors
$\zsnu$. Here, $W_{\nu}$ is a Whittaker function for $\pi_{\nu}$ and $W'_{\nu}$ a Whittaker function for $\pi'_{\nu}$, and $\zsnu$ may be seen to amount, essentially, to a certain Mellin transform of $W_{\nu}$ times $W'_{\nu}$. (More specific details on these local Whittaker functions and zeta integrals will be provided in the sections to follow.)

Again, the program at hand entails the comparison of $\lsnu$ to $\zsnu$, at each place $\nu$. To this end, it has been shown  that, if $\nu$ is a nonarchimedean place of $F$, and $\pi_{\nu}$ and
$\pi'_{\nu}$ are unramified, class one principal series representations, then 
$\lsnu$ and $\zsnu$ are, in each of the cases delineated above,  in fact {\it equal}.  (See the references cited in the fourth paragragh of this Introduction.)

In the present work we will show that, in certain circumstances, the
archimedean places behave like the unramified nonarchimedean ones, while in other circumstances they behave somewhat differently.
Specifically, let us and assume that $\nu$ is real or complex, and that
$\pi_{\nu}$ and $\pi'_{\nu}$ are irreducible class one principal series
representations.  (Irreducibility of such representations is automatic  for almost all values of $\nu$ and $\nu'$.)
We prove that, if these conditions are met, then in all of the above cases in which $\ell$ is nonnegative, we have
\begin{align*}
\zsnu=\lsnu. \end{align*} 
These results agree with a conjecture of Bump [Bu3], which states that, in general, archimedean zeta integrals and $L$-functions should   coincide in cases where the former are defined as integrals over the full torus.

Strictly speaking, our computations below will be performed under the assumption that $\nu$ is, in fact, {\it real}. In the cases where $ G =GL_{n+1-\ell}$ for some $\ell$, the real case implies the complex case, by straightforward relationships (cf. \cite{St95}) between real and complex Whittaker functions for $GL_m$. Analogous relationships between real and complex $SO_{2n+1}$ Whittaker functions may be deduced similarly.

We remark that, in the cases $(G,G')=(GL_n,GL_{n-1})$ and $(G,G')=(GL_n,GL_n)$, equality of the archimedean zeta integrals and the corresponding $L$-functions has already been demonstrated, under the somewhat more restrictive assumption that the representations in question are induced from representations with trivial central characters. See \cite{Ja2} for the case $ G \times G'=GL_2\times GL_2$, \cite{Bu1} for the case $ G \times G'=GL_3\times GL_2$, \cite{St01} for the general case of $GL_{n}\times GL_{n-1}$, and \cite{St02} for the general case of $GL_n\times GL_n$. By expressing Whittaker functions corresponding to more general central characters in terms of those arising in the case of trivial central characters, we obviate, in the present work, the need for such restrictions. Our proofs herein are further simplified by the use of more recently obtained recursive formulas, developed by the authors in \cite{ISt}, for archimedean Whittaker functions on $GL_n$.

We also note that, for certain small values of $n$, the archimedean calculations for $SO_{2n+1}\times GL_{n+\ell}  $ have also been developed previously. See \cite{N} for the case of $SO_5\times GL_2$, and \cite{I10} for the case of $SO_5\times GL_3$. (Results concerning $SO_3$ are subsumed by those regarding $GL_2$, since $SO_3(\R)$ Whittaker functions may be realized as special cases of $GL_2(\R)$ Whittaker functions. In particular, compare Proposition 1.2, below, in the case $n=2$, with Proposition 1.3 in the case $n=1$.)

In the cases of $ GL_n\times GL_{n-2}$ and $ SO_{2n+1}\times GL_{n-1}$, the archimedean zeta integrals in question are taken over proper subgroups of the full torus. In such situations one does not, generally, expect equality of these zeta integrals with the associated local $L$-functions. Indeed Hoffstein and Murty \cite{HM} have shown that,   in the case of $GL_3\times GL_1 $ and for $\nu$ real, the quotient
$$\zsnu/\lsnu
$$is expressible as a single-fold Barnes integral (as defined in Section 1.3 below). In this present work, we extend the results of Hoffstein and Murty to the case of general $n$, and give explicit formulas for the Barnes integrals that arise in such cases. Our results here may also be seen as analogous to those of the first author of this work and of Moriyama \cite{IM}, for the spinor $L$-function on $GSp_2$.

The present paper will proceed as follows.
In Section 1 we recall fundamental notions of principal series representations; we also present a variety of recursive formulas (many of them formerly derived elsewhere, but a number of them new) for class one Whittaker functions, and for Mellin transforms thereof. We also derive some recursive formulas for (multifold) Barnes integrals, which will be critical to our evaluation of archimedean zeta integrals in subsequent sections.

In Section 2 we consider the archimedean zeta integrals for $GL_n \times  GL_{n-1+\ell}$, in the cases $\ell=-1,0,1$. We demonstrate that, in the last two of these cases, these zeta integrals coincide with the associated archimedean $L$-functions, which are themselves expressible as products of $(n-1+\ell)n$ Gamma functions. Again, these results are consistent with Bump's conjecture \cite{Bu2} concerning archimedean zeta integrals over a full torus. We also show that, in the first case (where the integration is {\it not} over the full torus), the archimedean zeta integral equals the associated $L$-factor (which is in turn equal to a product of $(n-2)n$ Gamma functions) {\it times} a certain Barnes-type integral (in a single variable of integration). As noted above, this generalizes work of Hoffstein and Murty  concerning $L$-functions on $GL_3\times GL_1$, and parallels work of the first author of this work and of Moriyama  regarding the spinor $L$-function on $GSp_2$.

In Section 3 we compute the archimedean zeta integrals for $SO_{2n+1}\times GL_{n+\ell} $, for $\ell$ equal to $-1$, $0$, or $1$. We demonstrate that, in the last two of these cases, the zeta integral coincides with the associated archimedean $L$-factor, which itself equals a product of $2n(n+\ell)  $ Gamma factors, divided by a product of $n+\ell$ such. (These results again reflect Bump's conjecture, cited directly above, and generalize the work of Niwa \cite{N} in the case of $SO_5\times GL_2$, and of the first author of this paper \cite{I10} for the case of $ SO_5\times GL_3  $.) Additionally we demonstrate that, in the  case of $ SO_{2n+1}\times GL_{n-1}$ (where the integration is {\it not} over the full torus), the archimedean zeta integral equals the associated $L$-factor (which comprises, in this case,  $2n(n-1)$ Gamma factors in the numerator and $n-1$ in the denominator)  {times} a single-fold integral of Barnes type.

Finally, in our appendix, we outline an approach to unramified zeta integrals on $SO_{2n+1}\times GL_{n} $  and on $SO_{2n+1}\times GL_{n+1} $ that mimics, in number of ways, our treatment of the analogous archimedean entities.  While these unramified calculations have been performed previously, in more direct fashions, we hope that the parallels suggested by our method might help elucidate other aspects of the  archimedean theory.

The authors would like to thank Professor Takayuki Oda for his time, his expertise, his hospitality, and his generosity, all of which greatly facilitated the collaboration that resulted in this work.

\section{Explicit formulas for Whittaker functions}

In this section, we discuss class one principal series representations and the associated Whittaker functions.  We also recall explicit formulas, obtained by the authors \cite{I08}, \cite{ISt}, for Whittaker functions
on $ GL_n(\R) $ and $ SO_{2n+1}(\R) $, and provide some new formulas concerning $ GL_n(\R) $ Whittaker functions.

\subsection{Class one Whittaker functions}

Let $ G $ be the group $ GL_n(\R) $ or $ SO_{2n+1}(\R) $.
Here $ SO_{2n+1}(\R) = SO_{n+1,n}(\R) $ is the split special orthogonal group
with respect to the anti-diagonal matrix
$ \begin{pmatrix}
& & 1 \\ & \rotatebox[origin=c]{45}{$\cdots$} & \\ 1 & &
\end{pmatrix} $
of size $2n+1$.
Denote by $X$ the maximal unipotent subgroup of $G$ that consists of
upper triangular unipotent matrices in $G$.
Let $ Y $ be the subgroup of $G$ defined by
\begin{align*}
Y  = \begin{cases}
\{ {\rm diag}(t_1,\dotsc,t_n) \mid t_i >0 \} & \mbox{ if } G = GL_n(\R), \\
\{ {\rm diag}(t_1,\dotsc,t_n,1,t_n^{-1},\dotsc,t_1^{-1})
\mid t_i>0 \} & \mbox{ if } G= SO_{2n+1}(\R).
\end{cases}
\end{align*}
We introduce coordinates $ (y_1,\dotsc,y_n) $ $ (y_i>0) $ on $Y$ by putting
$ y_i  = t_i/t_{i+1} $ for $ 1 \le i \le n-1 $ and $ y_n  = t_n $; we then write
\begin{align*}
\alpha[y_1,\dotsc,y_n] & = {\rm diag}(y_1 \cdots y_n, y_2 \cdots y_n, \dotsc, y_n) \in GL_n(\R),
\\
\beta[y_1,\dotsc,y_n] & = {\rm diag}(y_1 \cdots y_n, \dotsc, y_n,1,y_n^{-1}, \dotsc,
(y_1 \cdots y_{n})^{-1}) \in SO_{2n+1}(\R).
\end{align*}

We write $ O(n) = \{ g \in GL_n(\R) \mid {}^tg g = 1 \} $ and $ SO(m) = SL_m(\R) \cap O(m) $, and take $K$ to be the maximal compact subgroup of $ G $ defined by
\begin{align*}
K = \begin{cases}
O(n) & \mbox{ if } G = GL_n(\R),
\\ SO(2n+1) \cap G \cong SO(n) \times SO(n+1) & \mbox{ if } G =SO_{2n+1}(\R). \end{cases}\end{align*}
Then the Iwasawa decomposition $ G = XYK $ holds.

To introduce the notion of a Whittaker function, we define a nondegenerate unitary character $ \psi $ on $X $ by
\begin{align*}
\psi(x)  = \exp \Bigl( 2\pi i \sum_{j=1}^{m} x_{j,j+1} \Bigr), \ \ \ x = (x_{j,k}) \in X.
\end{align*}
Here $ m = n-1$ if $ G=GL_n(\R) $, while $ m = n $ if $ G=SO_{2n+1}(\R) $.
Let $ {\mathcal W}(\psi) $ be the space of smooth functions $w : G \to \C $ satisfying
$$
w(xg) = \psi(x) w(g), \ \ \mbox{for all } (x,g) \in X \times G.
$$
By the right regular action, the space $ {\mathcal W}(\psi) $ becomes a $G$-module.
For an irreducible admissible smooth representation $ \pi $ of $ G $,
it is known that the dimension of the space
$ {\rm Hom}_{G}(\pi, {\mathcal W}(\psi)) $
of intertwining operators is at most one. Throughout, we will assume that $\pi$ is {\it generic}, meaning this dimension is, in fact, exactly equal to one. (Again, this stipulation is redundant in the case of $GL_n(\R)$.)

We denote by $ {\mathcal W}(\pi,\psi) $ the image of $ \pi $ in
$ {\mathcal W}(\psi) $.
Then, for a vector $ v \in \pi $, we define the ($\psi$-){\it Whittaker function attached to $v$} to be the image
$ W_v $ (uniquely defined up to a scalar) of $v$ in $ {\mathcal W}(\pi,\psi) ,$ under such an
intertwining operator.

In this paper we consider Whittaker functions for {\it class one principal series representations} of $G$, defined as follows.
Let $ H $ be the Borel subgroup of $G$ with Langlands decomposition $H = MYX $,
where
\begin{align*}
M= \begin{cases}
\{ {\rm diag}(\varepsilon_1,\dotsc,\varepsilon_n) \mid \varepsilon_i \in \{\pm 1\} \}
& \mbox{ if } G = GL_n(\R), \\
\{ {\rm diag}(\varepsilon_1,\dotsc,\varepsilon_n,1,\varepsilon_n,\dotsc,\varepsilon_1)
\mid \varepsilon_i \in \{\pm 1\} \} & \mbox{ if } G= SO_{2n+1}(\R).
\end{cases}
\end{align*}
Let $$ y^{\rho ,A}  = \prod_{j=1}^{n-1} y_j^{j(n-j)/2},\quad y^{\rho, B}  = \prod_{j=1}^{n} y_j^{j(n-j/2)}.$$ 
Then for $ a  = (a_1,\dotsc,a_n), b  = (b_1,\dotsc,b_n) \in \C^n $,
we define characters $ \chi_a^A $, $\chi_b^B$ of $H$ by
\begin{align*}
\begin{cases}
\chi_a^A(myx) = y^{\rho ,A} \prod_{j=1}^n y_j^{a_1+\cdots+a_j } & \mbox{ if } G= GL_n(\R), \\
\chi_b ^B(myx) = y^{\rho ,B} \prod_{j=1}^n y_j^{b_1+\cdots+b_j  } & \mbox{ if } G= SO_{2n+1}(\R),
\end{cases}
\end{align*}
where $ m \in M $, $ x \in X $, and
\begin{align*}
Y\ni y =\begin{cases}
\alpha[y_1,\dotsc,y_n] & \mbox{ if } G= GL_n(\R), \\
\beta[y_1,\dotsc,y_n]
& \mbox{ if } G= SO_{2n+1}(\R).
\end{cases}
\end{align*}
We call the induced representations
$$ \pi_a^A  = {\rm Ind}_{H}^{GL_n(\R)}(\chi_{a}^A)
\mbox{ and }
\pi_b^{B}  = {\rm Ind}_{H}^{SO_{2n+1}(\R)} (\chi_{b }^{B}) $$
the {\it class one principal series representations} of $ GL_{n}(\R) $ and
$ SO_{2n+1}(\R) $, respectively.

There is then a vector $ v_0 $ in $ \pi_a^A $ (respectively, $\pi_b^B$), unique up to constants, satisfying $\pi_a^A(k)v_0=v_0$ (respectively, $\pi_b^B(k)v_0=v_0$) for all $k\in K$. We call the Whittaker function $ W_{v_0} $ attached to $ v_0 $
the {\it class one Whittaker function} on $G$. (This Whittaker function is uniquely defined only up to scalars; it is conventional to normalize our choice of $W_{v_0}$ by stipulating that $W_{v_0}(e)=1$, where $e$ is the identity in $G$.)
Since the class one Whittaker function $W_{v_0}$ on $G$ is a right $K$-invariant function,
the Iwasawa decomposition implies that
$ W_{v_0} $ is determined by its restriction $W_{v_0}|_Y $ to $Y$.
Conversely, we can extend a function $w$ on $ Y $ to a function on $G$ by defining
\begin{align} \label{radial}
w(g)  = \psi(x(g)) w(y(g)), \ \ g \in G,
\end{align}
where $ g = x(g) y(g) k(g) $ $ (x(g) \in X, y(g) \in Y, k(g) \in K) $ is
the Iwasawa decomposition of $g$.


\subsection{Integral representations of  class one Whittaker functions}

We now wish to present explicit formulas for the class one Whittaker functions on $G$.
Since these formulas relate Whittaker functions
on $GL_{n}(\R) $ and $SO_{2n+1}(\R) $ to those on $ GL_{n-1}(\R) $ and $SO_{2n-1}(\R) $
respectively, we need to first do the same for our class one principal series representations.

\begin{defn}
\begin{itemize}
\item[(1)]
To the class one principal series representation $ \pi_a^A $ of $GL_n(\R) $,
we associate a class one principal series
representation $ \pi_{\widetilde{a} } $ ($\widetilde{a}  = (\widetilde{a}_1, \dotsc, \widetilde{a}_{n-1})$) of $ GL_{n-1}(\R)$
by putting $\widetilde{a}_j  = a_{j+1} + a_1/(n-1) $ for $1\le j\le n-1$.
\item[(2)]
To the class one principal series representation $ \pi^B_{b} $ of $SO_{2n+1}(\R) $,
we associate a class one principal series representation
$ \pi^B_{\widetilde{b} } $ ($\widetilde{b}  = (\widetilde{b}_1, \dotsc, \widetilde{b}_{n-1})$) of $ SO_{2n-1}(\R)$
by putting $ \widetilde{b}_j  = b_j $ for $1\le j\le n-1$.
\end{itemize}
\end{defn}

Note that, if we write
$$
|a|  = \sum_{j=1}^n a_j, \ \ \
|\widetilde{a}|  = \sum_{j=1}^{n-1} \widetilde{a}_j
$$
for $ \pi_a^A $ and $ \pi_{\widetilde{a}}^A $, then have $ |a| = |\widetilde{a}| $.

Our explicit formulas for class one Whittaker functions in terms of
the associated class one principal series representations are given by the following two propositions.

\begin{prop} \label{glwhit}\cite[cf. Theorem 14]{ISt}
For $ G = GL_n(\R) $, we define a function
\begin{align*}
W_{n,a}^A (\alpha[y_1,\dotsc,y_n])
& = y^{\rho,A}\,
\widehat{W}_{n,a}^A(\alpha[y_1,\dotsc,y_{n-1},y_n])
= y^{\rho,A} \cdot y^{|a|}
\widehat{W}_{n,a}(\alpha[y_1,\dotsc,y_{n-1},1])
\end{align*}
on $ Y $ by the recursive relation
\begin{align*}
& \widehat{W}_{n,a}^{A}(\alpha[y_1,\dotsc, y_{n-1},1 ])
\\
&  = \int_{(\R_+)^{n-1}}
\widehat{W}_{n-1,\widetilde{a}}^A
\left(\alpha\left[
y_2 \sqrt{\frac{t_2}{t_1}},\dotsc,y_{n-1} \sqrt{\frac{t_{n-1}}{t_{n-2}}},1 \right] \right)
\\
& \times
\biggl[ \prod_{j=1}^{n-1} \exp \Bigl\{ -(\pi y_j)^2 t_j - \frac{1}{t_j} \Bigr\}
\cdot (\pi y_j)^{ {(n-j)a_1}/{(n-1)} }
t_j^{ {n a_1}/{(2(n-1))} } \biggr]
(\pi t_{n-1})^{- {|a|}/{2}}
\prod_{j=1}^{n-1} \frac{dt_j}{t_j}
\end{align*}for $n\ge3$,
and $ \widehat{W}_{2,(a_1,a_2)}^{A}(\alpha[y_1,1])  = 2y_1^{(a_1+a_2)/2} K_{(a_1-a_2)/2}(2\pi y_1) $. (Here, $K$ denotes the $K$-Bessel function, cf. \cite{WW}.)
If we extend the function $ W_{n,a}^A $ to $ G $ by $ (\ref{radial}) $,
then $ W_{n,a}^A $ gives a class one Whittaker function on $G$.
\end{prop}
\bpf
In our previous paper \cite{ISt} we considered the context of $ SL_n(\R) $, that is,
we assumed $ |a| = 0 $ for the parameter of
the class one principal series representation.
The recursive relation for the radial part
$$ \widehat{W}_{n,\nu}^{SL}(y_1,\dotsc,y_{n-1})
= \widehat{W}_{n,\nu}^{SL}( \alpha[y_1,\dotsc, y_{n-1},
(y_1 y_2^2 \cdots y_{n-1}^{n-1})^{-1/n}])
$$
of the ``$\rho$-shifted'' Whittaker function on $SL_n(\R) $ is given by
\begin{align*}
\widehat{W}_{n,\nu}^{SL}(y_1,\dotsc, y_{n-1})
& = \int_{(\R_+)^{n-1}}
\widehat{W}_{n-1,\widetilde{\nu}}^{SL}
\left(
y_2 \sqrt{\frac{t_2}{t_1}},\dotsc,y_{n-1} \sqrt{\frac{t_{n-1}}{t_{n-2}}}\right)
\\
& \times
\biggl[ \prod_{j=1}^{n-1} \exp \Bigl\{ -(\pi y_j)^2 t_j - \frac{1}{t_j} \Bigr\}
\cdot (\pi y_j)^{ {(n-j) \nu_1}/{(n-1)} }
t_j^{ {n \nu_1}/{(2(n-1))} } \frac{dt_j}{t_j} \biggr]
\end{align*}
(\cite[Theorem 14]{ISt}).
Here $ \nu  = (\nu_1,\dotsc,\nu_n) \in \C^n $ satisfies $ \sum_{i=1}^n \nu_i = 0 $,
and $ \widetilde{\nu}  = (\widetilde{\nu}_1,\dotsc,\widetilde{\nu}_{n-1}) $ with
$ \widetilde{\nu}_i  = \nu_{i+1} + \nu_1/(n-1) $.
In view of the formula
\begin{align*}
W_{n,a}^A(\alpha[y_1,\dotsc,y_n])
& = (\det \alpha[y_1,\dotsc, y_n])^{|a|/n} W_{n,a}^A(\alpha[y_1,\dotsc, y_{n-1},
(y_1 y_2^2 \cdots y_{n-1}^{n-1})^{-1/n}]),
\end{align*}
we find, writing $ \nu_i  = a_i - |a|/n $ for $ 1\le i \le n $, that
\begin{align} \label{SLGL}
\widehat{W}_{n,a}^A(\alpha[y_1,\dotsc,y_{n-1},1])
& = \biggl[ \prod_{j=1}^{n-1} y_j^{ {j|a|}/{n} } \biggr]
\widehat{W}_{n,\nu}^{SL}(y_1,\dotsc, y_{n-1}).
\end{align}
From this, and the fact that $ |\widetilde{a}| = |a| $, we get
\begin{align*}
\widehat{W}_{n,a}^{A}(\alpha[y_1,\dotsc, y_{n-1},1 ])
& = \prod_{j=1}^{n-1} y_j^{ {j|a|}/{n} }
\int_{(\R_+)^{n-1}}
\prod_{j=1}^{n-2} \biggl( y_{j+1} \sqrt{ \frac{t_{j+1}}{t_j} } \biggr)^{- {j|a|}/{(n-1)} }
\\
& \times
\widehat{W}_{n-1,\widetilde{a}}^A
\left(\alpha\left[
y_2 \sqrt{\frac{t_2}{t_1}},\dotsc,y_{n-1} \sqrt{\frac{t_{n-1}}{t_{n-2}}},1 \right] \right)
\\
& \times
\biggl[ \prod_{j=1}^{n-1} \exp \Bigl\{ -(\pi y_j)^2 t_j - \frac{1}{t_j} \Bigr\}
\cdot (\pi y_j)^{ {(n-j)\nu_1}/{(n-1)} }
t_j^{ {n \nu_1}/{(2(n-1))} } \frac{dt_j}{t_j} \biggr].
\end{align*}
We can check that
\begin{align*}
& \biggl[ \prod_{j=1}^{n-1} y_j^{ {j|a|}/{n} } \biggr]
\biggl[ \prod_{j=1}^{n-2} \biggl( y_{j+1} \sqrt{ \frac{t_{j+1}}{t_j} } \biggr)^{- {j|a|}/{(n-1)} }
\biggr]
\biggl[ \prod_{j=1}^{n-1} (\pi y_j)^{ {(n-j)\nu_1}/{(n-1)} }
t_j^{ {n \nu_1}/{(2(n-1))}} \biggr]
\\
& =
\biggl[ \prod_{j=1}^{n-1} (\pi y_j)^{ {(n-j)a_1}/{(n-1)} }
t_j^{ {n a_1}/{(2(n-1))}} \biggr] \cdot (\pi t_{n-1})^{- {|a|}/{2} }.
\end{align*} So our recursive formula in \cite{ISt} yields the recursive formula in the statement of the present proposition, and
we are done.
\epf
\begin{prop} \cite[Theorem 3.1]{I08}\label{sowhit}
  For $ G=SO_{2n+1}(\R) $, we define a function
$$
W_{n,b}^B (\beta[y_1,\dotsc,y_n])
 = y^{\rho, B}\,
\widehat{W}_{n,b}^B(\beta[y_1,\dotsc,y_{n}]) $$
on $ Y $ by the recursive relation
\begin{align*}
& \widehat{W}_{n,b}^B (\beta[y_1,\dotsc,y_n])
\\
&
 = \int_{(\R_+)^{n}}
\int_{(\R_+)^{n-1}}
\widehat{W}_{n-1,\widetilde{b} }^{B}
\left( \beta\left[ y_2 \sqrt{\frac{t_2 u_2}{t_3 u_1}}, \dotsc,
y_{n-1} \sqrt{\frac{t_{n-1} u_{n-1}}{ t_n u_{n-2}}},
y_n \sqrt{ \frac{t_n}{u_{n-1}} } \, \right] \right)
\\
& \times
\biggl[ \prod_{j=1}^{n} \exp \biggl\{ -(\pi y_j)^2 t_j - \frac{1}{t_j} \biggr\}
\cdot (\pi y_j)^{b_n} \biggr]
\biggl[ \prod_{j=1}^{n-1}
\exp \biggl\{ -(\pi y_j)^2 \frac{t_j}{t_{j+1}} u_j - \frac{1}{u_j} \biggr\}
\cdot (t_{j+1} u_j)^{ {b_n}/{2}} \biggr]
\\
& \times t_1^{b_n}
\prod_{j=1}^{n-1} \frac{du_j}{u_j} \prod_{j=1}^{n} \frac{dt_j}{t_j}
\end{align*}for $n\ge2$, and
$ \widehat{W}_{1,(b_1)}^{B}(\beta[y_1])  = 2K_{b_1/2}(2\pi y_1) $.
If we extend the function $ W_{n,b}^B $ to $ G $ by
$ (\ref{radial}) $,
then $ W_{n,b}^B $ gives a class one Whittaker function on $G$.
\end{prop}

For later use, we derive a new, additional recursive formula for $GL(n,\R)$ class one Whittaker functions.
\begin{prop} \label{GLrec2} For $s\in\C$, define\begin{align}\Gamma_\R(s) =\pi^{-s/2} \Gamma(s/2).\end{align}
We then have the recurrence relation
\begin{align} \label{GLrec2int}
\begin{split}
\widehat{W}_{n,a}^{A}(\alpha[y_1,\dotsc,y_n])
& = \biggl[\prod_{j=2}^{n} \Gamma_{\R}(a_1-a_j+1) \biggr]
\cdot y_n^{|a|} \biggl[ \prod_{j=1}^{n-1} y_j^{- {ja_1}/{(n-1)}} \biggr]
\\
& \times \int_{\R^{n-1}} \widehat{W}_{n-1,\widetilde{a}}^A
\biggl( \alpha \biggl[ y_1 \frac{\sqrt{q_1 q_3}}{q_2} ,  y_{ 2} \frac{\sqrt{q_{ 2}q_4}}{q_3},\dotsc,
y_{n-1} \frac{\sqrt{q_{n-1}q_{n+1}}}{q_n}\biggr] \biggr)
\\
& \times \exp \biggl\{-2\pi i \biggl(
\sum_{j=1}^{n-1} \frac{x_j x_{j+1}}{q_{j+1}} y_j  \biggr) \biggr\}
\cdot q_1^{- {na_1}/{(2(n-1))}} \prod_{j=1}^{n-1} \frac{dx_j}{\sqrt{q_j}},
\end{split}
\end{align}
where
$$
q_j \equiv q_j(x_1,\dotsc,x_{n-1})  = 1 + \sum_{p=j}^{n-1} x_p^2 \quad \hbox{\rm for } 1 \le j \le n-1 ;\quad q_n=q_{n+1}=x_n =1.$$

\end{prop}
\bpf
To prove the relation (\ref{GLrec2int}), we recall the original integral representation
of class one $GL(n,\R)$ Whittaker functions, introduced by Jacquet \cite{Ja1}:
\begin{align*}
J_{n,a}(g) = \int_X H_{n,a} (w_n xg) \psi^{-1}(x) \, dx, \ \ \
w_n = \begin{pmatrix}
& & 1 \\ & \rotatebox[origin=c]{45}{$\cdots$} & \\ 1 & &
\end{pmatrix},
g \in GL_n(\R).
\end{align*}
Here $ H_{n,a}(g)  = \chi_{a}(y(g)) = y^{\rho, A}\prod_{j=1}^{n} y_j^{a_1+\cdots+a_j } $
if $ y(g) = \alpha[y_1,\dotsc, y_n] $.
Then, as is shown in \cite{St90}, $J_{n,a}(g)$ is related to the class one Whittaker function $ W_{n,a}^A  $ defined above by the formula
\begin{align} \label{WvsJ}
W_{n,a}^A(g)
= \biggl[ \prod_{1 \le j<k \le n} \Gamma_{\R}(a_j-a_k+1) \biggr]
J_{n,a}(g).
\end{align}
For $x\in X$ as in the Jacquet integral, let use write $ x  = (x_{j,k})_{1 \le j,k \le n}$.
As in \cite[\S 3]{St90}, the substitutions
$ x_{j,k} \to (\prod_{p=j}^{k-1} y_p) x_{j,k} $ $ (1\le j<k \le n) $ imply that
$$
J_{n,a}(\alpha[y_1,\dotsc, y_n]) = y^{\rho,A} \widehat{J}_{n,a}(\alpha[y_1,\dotsc,y_n]),
$$
where
\begin{align} \label{Jac}
\begin{split}
\widehat{J}_{n,a}(\alpha[y_1,\dotsc,y_n])
&  = y_n^{|a|} \biggl[ \prod_{j=1}^{n-1} y_j^{-(a_1+a_2+\cdots+a_{n-j})+|a|} \biggr]
\\
& \times \int_{\R^{n(n-1)/2}} H_{n,a}(w_n x)
\cdot \exp \biggl(-2\pi i \sum_{j=1}^{n-1} x_{j,j+1} y_j \biggr)
\prod_{1 \le j<k \le n} dx_{j,k}.
\end{split}
\end{align}

The following explicit formula for $ H_{n,a}(w_nx) $ given in \cite[Appendix]{St90}:
\begin{align} \label{H}
H_{n,a}(w_nx) = \prod_{m=1}^{n-1} (\Delta_{n,m} (x))^{ (a_{n-m+1} - a_{n-m} -1)/2},
\end{align}
where $ \Delta_{n,m} (x) $ is the
sum of the squares of the $ m \times m $ subdeterminants of the $ m \times n $ matrix formed from 
the top $ m $ rows of $x$.
Using this formula, we will modify the integral (\ref{Jac}) to prove our proposition.
Roughly speaking, the idea here is that that, if we integrate the integrand in  (\ref{Jac}) over those $x_{j,k}$ with
  $ 1\le j<k \le n-1 $ only, we are left with a certain integral involving $ \widehat{J}_{n-1,\widetilde{a}} $.

We first relate $ H_{n,a} $ to $ H_{n-1,\widetilde{a}} $.

\medskip

\noindent
{\bf Claim.} {\it For $ x = (x_{j,k})_{1 \le j,k \le n} \in X $,
let introduce new variables
$ x_j $ $ (1 \le j \le n-1) $, and   an $(n-1)\times(n-1)$ upper triangular unipotent matrix
$ x' = (x_{j,k}')_{1\le j,k \le n-1} $, by the relations
\begin{align*}
x_{j,n} = \sum_{\ell=j}^{n-1} x_{j,\ell} x_\ell \ \ \ (1 \le j \le n-1),
\end{align*} and
\begin{align*}
x_{j,k} = \sqrt{ \frac{q_j q_{k+1}}{q_{j+1} q_k}} \biggl( x_{j,k}'
- \sum_{\ell=j}^{k-1} x_\ell x_k \sqrt{\frac{q_k}{q_\ell q_{\ell+1} q_{k+1}} } x_{j,\ell}' \biggr)
\ \ \ (1 \le j<k \le n-1).
\end{align*}
Then we have
\begin{align*}
H_{n,a}(w_n x) = \biggl[ \prod_{m=1}^{n-1}
\biggl( \frac{q_1}{q_{m+1}} \biggr)^{(a_{n-m+1}-a_{n-m}-1)/2} \biggr]
H_{n-1,\widetilde{a}} (w_{n-1} x').
\end{align*} }
\medskip

{\it Proof of Claim.}
We can write $ x = \widetilde{x} \kappa $, with
\begin{align*}
\widetilde{x} = \begin{pmatrix} 0 & x'' \\ q_1^{-1/2} & \overline{x} \end{pmatrix},
\ \ \
\kappa = \begin{pmatrix} \overline{\kappa} & q_1^{-1/2} \\ \kappa' & {}^t \overline{x} 
\end{pmatrix}.
\end{align*}
Here
\begin{align*}
x''& = \bigl( (q_j/q_{j+1})^{1/2} x_{j,k}' \bigr)_{1 \le j,k \le n-1},
\\
\overline{x} & = \Bigl( \frac{x_1}{\sqrt{q_1 q_2}}, \frac{x_2}{\sqrt{q_2q_3}},
\dotsc, \frac{x_{n-2}}{\sqrt{q_{n-2}q_{n-1}}}, \frac{x_{n-1}}{\sqrt{q_{n-1}}} \Bigr),
\\
\overline{\kappa}
&= \Bigl( -\frac{x_1}{\sqrt{q_1}}, -\frac{x_2}{\sqrt{q_1}}, \dotsc, -\frac{x_{n-1}}{\sqrt{q_1}} \Bigr),
\end{align*}
and $ \kappa' = (\kappa_{j,k}')_{1 \le j,k \le n-1} $
with
\begin{align*}
\kappa_{j,k}'
= \begin{cases}
(q_{j+1}/q_j)^{1/2} & \mbox{ if } 1 \le j=k \le n-1, \\
-(q_j q_{j+1})^{-1/2} x_j x_k & \mbox{ if } 1 \le j<k \le n-1, \\
0 & \mbox{ otherwise. }
\end{cases}
\end{align*}
We check that $ \kappa \in K $.  Then, since $ H_a $ is right $K$-invariant, we have
\begin{align*}
H_{n,a}(w_nx)
& = H_{n,a}(w_n \widetilde{x}) \\
& = \prod_{m=1}^{n-1}
( \Delta_{n,m}(\widetilde{x}) )^{(a_{n-m+1}-a_{n-m}-1)/2} \\
& = \prod_{m=1}^{n-1}
\biggl( \frac{q_1}{q_{m+1}} \cdot \Delta_{n-1,m}(x') \biggr)^{\!(a_{n-m+1}-a_{n-m}-1)/2}
\\
& = \biggl[ \prod_{m=1}^{n-1} \biggl( \frac{q_1}{q_{m+1}} \biggr)^{\! (a_{n-m+1}-a_{n-m}-1)/2} \biggr]
H_{n-1,\widetilde{a}} (w_{n-1} x').
\end{align*}
Here we used (\ref{H}) and the relation
$ \widetilde{a}_{n-m+1}-\widetilde{a}_{n-m} = a_{n-m+1}-a_{n-m} $.
This proves the claim.

\medskip

Returning to the integral (\ref{Jac}), we change variables from  $ x_{j,k} $  ($1\le j<k\le n$) to  
$  x_{j,k}' $ ($1\le j<k\le n-1$) and  $ x_j $  ($1\le j\le n-1$); 
we further substitute $ x_j \to (-1)^{n-j-1} x_j $ ($1\le j\le n-1$).
Then we get
\begin{align*}
\widehat{J}_{n,a}(\alpha[y_1,\dotsc,y_n])
& = y_n^{|a|} \biggl[ \prod_{j=1}^{n-1} y_j^{-(a_1+\cdots+a_{n-j})+|a|} \biggr]
\\
& \times \int_{\R^{n(n-1)/2}}
H_{\widetilde{a}}(w_{n-1} x)
\cdot \exp \biggl( -2\pi i \sum_{j=1}^{n-2} \frac{ \sqrt{q_j q_{j+2}} }{q_{j+1}}
x_{j,j+1}' y_j \biggr)
\\
& \times \exp \biggl\{ -2\pi i \biggl( \sum_{j=1}^{n-1} \frac{x_j x_{j+1}}{q_{j+1}} y_j
  \biggr) \biggr\}
\\
& \times
\biggl[ \prod_{j=1}^{n-1} \biggl( \frac{q_1}{q_{j+1}} \biggr)^{(a_{n-j+1}-a_{n-j}-1)/2} \biggr]
\biggl[ \prod_{1 \le j<k \le n-1} \sqrt{ \frac{q_j q_{k+1}}{q_{j+1} q_k} } \, \biggr]
\\
& \times
\prod_{1 \le j<k \le n-1} dx_{j,k}' \prod_{j=1}^{n-1} dx_j.
\end{align*}
Now we can use (\ref{Jac}) to integrate with respect to $ x_{j,k}' $; we find that
\begin{align*}
\widehat{J}_{n,a}(\alpha[y_1,\dotsc,y_n])
& = y_n^{|a|} \biggl[ \prod_{j=1}^{n-1} y_j^{-(a_1+\cdots+a_{n-j})+|a|} \biggr]
\cdot y_{n-1}^{-|a|}
\biggl[ \prod_{j=1}^{n-2} y_j^{(\widetilde{a}_1+\cdots+\widetilde{a}_{n-j-1})-|a|} \biggr]
\\
& \times \int_{\R^{n-1}}
\widehat{J}_{n-1,\widetilde{a}} \biggl(\alpha \biggl[ y_1 \frac{\sqrt{q_1 q_3}}{q_2}, 
y_{ 2}\frac{\sqrt{q_{ 2}q_4}}{q_{3}},\dotsc, y_{ n-1}\frac{\sqrt{q_{ n-1}q_{n+1}}}{q_{n}} \biggr] \biggr)
\\
& \times \exp \biggl\{ -2\pi i \biggl( \sum_{j=1}^{n-1} \frac{x_j x_{j+1}}{q_{j+1}} y_j \biggr) \biggr\}
\\
& \times
\biggl[ \prod_{j=1}^{n-1} \biggl( \frac{q_1}{q_{j+1}} \biggr)^{(a_{n-j+1}-a_{n-j}-1)/2} \biggr]
\biggl[ \prod_{1 \le j<k \le n-1} \sqrt{ \frac{q_j q_{k+1}}{q_{j+1} q_k} } \,\biggr]
\\
& \times
q_{n-1}^{-\frac{|a|}{2}}
\biggl[ \prod_{j=1}^{n-2} \biggl( \frac{\sqrt{q_j q_{j+2}} }{q_{j+1}}
\biggr)^{(\widetilde{a}_1+\cdots+\widetilde{a}_{n-j-1})-|a| } \biggr]
\prod_{j=1}^{n-1} dx_j.
\end{align*}
By collecting the powers of $ y_j $ and $ q_j $, we arrive at
\begin{align*}
& \widehat{J}_{n,a}(\alpha[y_1,\dotsc,y_n])
\\
& = y_n^{|a|} \biggl[ \prod_{j=1}^{n-1} y_j^{- {ja_1}/{(n-1)}} \biggr]
\int_{\R^{n-1}}
\widehat{J}_{n-1,\widetilde{a}} \biggl( \alpha \biggl[ y_1 \frac{\sqrt{q_1 q_3}}{q_2}, 
y_{2} \frac{\sqrt{q_{2}q_4}}{q_{3}},\dotsc, y_{n-1}\frac{ \sqrt{q_{n-1} q_{n+1}}}{q_n} \biggr] \biggr)
\\
& \times \exp \biggl\{ -2\pi i \biggl( \sum_{j=1}^{n-1} \frac{x_j x_{j+1}}{q_{j+1}} y_j
 \biggr) \biggr\}
\cdot q_1^{- {na_1}/{(2(n-1))}-1/2} \prod_{j=2}^{n-1} q_j^{-1/2}
\prod_{j=1}^{n-1} dx_j.
\end{align*}
The relation (\ref{WvsJ}) then completes the proof of this proposition.
\epf

\subsection{Mellin transforms of  class one Whittaker functions}
Here we recall certain  {\it Mellin-Barnes-type} representations of class one Whittaker functions.  By such a representation we mean, essentially, an  integral (perhaps multifold) over vertical lines in the complex plane, the integrand being a rational expression in Gamma functions times a product of powers of the independent variables in question, and the paths of integration being indented, if necessary, to separate increasing and decreasing sequences of poles.

For complex numbers $ s_j $,
let $ U_{n,a}(s_1,\dotsc,s_{n-1}) $ and
$ V_{n,b}(s_1,\dotsc,s_n) $ be Mellin transforms of
$ \widehat{W}_{n,a}^A(y) $ and $ \widehat{W}_{n,b}^B(y) $  respectively:
\begin{align*}
U_{n,a}(s_1,\dotsc,s_{n-1})
= \int_{(\R_+)^{n-1}} \widehat{W}_{n,a}^A(\alpha[y_1,\dotsc,y_{n-1},1])
\biggl[ \prod_{j=1}^{n-1} y_j^{s_j} \frac{dy_j}{y_j} \biggr],
\end{align*}
and
\begin{align*}
V_{n,b}(s_1,\dotsc,s_n)
= \int_{(\R_+)^n} \widehat{W}_{n,b}^B(\beta[y_1,\dotsc,y_n])
\biggl[ \prod_{j=1}^{n} y_j^{s_j} \frac{dy_j}{y_j} \biggr].
\end{align*}
Then Mellin inversion, and the definitions of $  {W}_{n,a}^A(y) $ and $  {W}_{n,b}^B(y) $ in terms of $ \widehat{W}_{n,a}^A(y) $ and $ \widehat{W}_{n,b}^B(y) $ (cf. Propositions \ref{glwhit} and \ref{sowhit}),  imply that
\begin{align*}
W_{n,a}^A(\alpha[y_1,\dotsc, y_n])
& = y_n^{|a|} \cdot
\frac{y^{\rho,A} }{(2\pi i)^{n-1}} \int_{s_1,\dotsc,s_{n-1}}
U_{n,a}(s_1,\dotsc,s_{n-1})
\biggl[ \prod_{j=1}^{n-1} y_j^{-s_j} ds_j \biggr],
\\
W_{n,b}^B(\beta[y_1,\dotsc,y_n])
& = \frac{y^{\rho, B} }{(2\pi i)^{n}} \int_{s_1,\dotsc,s_n}
V_{n,b}(s_1,\dotsc, s_{n})
\biggl[ \prod_{j=1}^{n} y_j^{-s_j} ds_j \biggr].
\end{align*}
We now recall recursive relations for the Mellin transforms $ U_{n,a} $ and $ V_{n,b} $.

\begin{prop} \cite[Theorem 12]{ISt} \label{MBGL}
We have
\begin{align*}
U_{n,a}(s_1,\dotsc,s_{n-1})
&= \frac{2^{-1} }{(4\pi i)^{n-2}} \int_{z_1,\dotsc,z_{n-2}}
U_{n-1, \widetilde{a} }(z_1,\dotsc,z_{n-2})
\\
& \times
\biggl[ \prod_{j=1}^{n-1} \Gamma_{\R} \Bigl( s_j-z_j- \frac{ja_1}{n-1} \Bigr)
\Gamma_{\R} \Bigl( s_j-z_{j-1} + \frac{(n-j)a_1}{n-1} \Bigr) \biggr]
\\
& \times
dz_1 \cdots dz_{n-2},
\end{align*}
where we understand that $ z_0 = 0 $ and $ z_{n-1} = -|a| $.
Note that $ U_{2,(a_1,a_2)}(s) = 2^{-1} \Gamma_{\R}(s+a_1)\Gamma_{\R}(s+a_2) $.
\end{prop}

\bpf
We use the same notation as in the proof of Proposition \ref{glwhit}, and
set
$$
U_{n, {\nu}}^{SL} (s_1,\dotsc, s_{n-1})
= \int_{(\R_+)^{n-1}} \widehat{W}_{n,\nu}^{SL}(y_1,\dotsc,y_{n-1})
\biggl[ \prod_{j=1}^{n-1} y_j^{s_j} \frac{dy_j}{y_j} \biggr].
$$
Then \cite[Theorem 12]{ISt} asserts that
\begin{align*}
U_{n,\nu}^{SL} (s_1,\dotsc,s_{n-1})
& = \frac{2^{-1}}{(4\pi i)^{n-2}} \int_{z_1,\dotsc,z_{n-2}}
U_{n,\widetilde{\nu}}^{SL}(z_1,\dotsc,z_{n-2})
\\
& \times
\biggl[ \prod_{j=1}^{n-1} \Gamma_{\R} \Bigl( s_j-z_j- \frac{j \nu_1}{n-1} \Bigr)
\Gamma_{\R} \Bigl( s_j-z_{j-1} + \frac{(n-j) \nu_1}{n-1} \Bigr) \biggr]
\\
& \times
dz_1 \cdots dz_{n-2},
\end{align*}
with $ z_0 = z_{n-1} = 0 $.
This, together with the formula
\begin{align*}
U_{n,a}(s_1,\dotsc, s_{n-1})
= U_{n,\nu}^{SL} \Bigl(s_1+\frac{|a|}{n}, s_2+\frac{2|a|}{n}, \dotsc,
s_{n-1} + \frac{(n-1)|a|}{n} \Bigr),
\end{align*}which follows immediately from  (\ref{SLGL}), then yield the desired result.
\epf

\medskip

For some investigations below in the case $G'=GL_{n-2}$, we present the following result concerning the  Whittaker function contragredient to $ W_{n,a} $.

\begin{cor} \label{GLcont}
For a $\psi$-Whittaker function $W$ on $ GL_n(\R) $, we put
$$ W^{\vee}(g) = W(w_n {}^t g^{-1}).
$$
Then we have
$$ (W_{n,a}^A)^{\vee}(\alpha[y_1,\dotsc,y_n]) = W_{n,-a}^A(\alpha[y_1,\dotsc,y_n]) $$
and
$ (W_{n,a}^A)^{\vee} \in \mathcal{W}(\pi_{-a}^A, \psi^{-1}) $.
\end{cor}

\bpf
We first prove
\begin{align} \label{U}
U_{n,a}(s_{n-1}-|a|,\dotsc, s_1-|a|) = U_{n,-a}(s_1,\dotsc,s_{n-1})
\end{align}
by induction on $n$. The case $n=2$ follows automatically from the definitions.  Now, by Proposition \ref{MBGL} and the induction hypothesis, we have
\begin{align*}
& U_{n,a}(s_{n-1}-|a|,\dotsc,s_1-|a|)
\\
& = \frac{2^{-1}}{(4\pi i)^{n-2}} \int_{z_1,\dotsc,z_{n-2}}
U_{n-1,-\widetilde{a}}(z_{n-2}+|\widetilde{a}| ,\dotsc, z_{1}+|\widetilde{a}|)
\\
& \times \biggl[ \prod_{j=1}^{n-1} \Gamma_{\R}\Bigl(s_{n-j}-|a|-z_j-\frac{ja_1}{n-1} \Bigr)
\Gamma_{\R} \Bigl( s_{n-j}-|a|-z_{j-1}+\frac{(n-j)a_1}{n-1} \Bigr) \biggr]
\\
& \times dz_1 \cdots dz_{n-2}.
\end{align*}
We substitute $ z_j \to z_{n-j-1}-|\widetilde{a}| = z_{n-j-1}-|a| $, to find that
\begin{align*}
& U_{n,a}(s_{n-1}-|a|,\dotsc,s_1-|a|)
\\
& = \frac{2^{-1}}{(4\pi i)^{n-2}} \int_{z_1,\dotsc,z_{n-2}}
U_{n-1,-\widetilde{a}}(z_1,\dotsc,z_{n-2})
\\
& \times \biggl[ \prod_{j=1}^{n-1}
\Gamma_{\R}\Bigl(s_{n-j}-z_{n-j-1}+|a|-\frac{ja_1}{n-1} \Bigr)
\Gamma_{\R} \Bigl(s_{n-j}-z_{n-j}+\frac{(n-j)a_1}{n-1} \Bigr) \biggr]
dz_1 \cdots dz_{n-2}
\\
& = \frac{2^{-1}}{(4\pi i)^{n-2}} \int_{z_1,\dotsc,z_{n-2}}
U_{n-1,-\widetilde{a}}(z_1,\dotsc,z_{n-2})
\\
& \times \biggl[ \prod_{j=1}^{n-1}
\Gamma_{\R}\Bigl(s_{j}-z_{j-1}+|a|-\frac{(n-j)a_1}{n-1} \Bigr)
\Gamma_{\R} \Bigl(s_{j}-z_{j}+\frac{ja_1}{n-1} \Bigr) \biggr]
dz_1 \cdots dz_{n-2}
\\
& = U_{n,-a}(s_1,\dotsc,s_{n-1}).
\end{align*}
Thus we obtain (\ref{U}).

Now, since $ W_{n,a}^A $ is right $ O(n) $-invariant, we have
\begin{align*}
& (W_{n,a}^A)^{\vee}(\alpha[y_1,\dotsc,y_n])
\\
& = (W_{n,a}^A)^{\vee}( w_n {}^t \alpha[y_1,\dotsc, y_n]^{-1} w_n)
\\
& = (W_{n,a}^A)^{\vee}(\alpha[y_{n-1},y_{n-2},\dotsc,y_1,(y_1\cdots y_n)^{-1}])
\\
& = \biggl[\prod_{j=1}^{n} y_j^{-|a|} \biggr]
\frac{ y^{\rho,A} }{ (2 \pi i)^{n-1} }
\int_{s_1,\dotsc, s_{n-1}} U_{n,a}(s_1,\dotsc,s_{n-1})
\biggl[ \prod_{j=1}^{n-1} y_{n-j}^{-s_j} \biggr] ds_1 \cdots ds_{n-1}.
\end{align*}
We substitute $ s_j \to s_{n-j}-|a| $ for $ 1 \le j \le n-1 $
and use (\ref{U}), to find that
\begin{align*}
& (W_{n,a}^A)^{\vee}(\alpha[y_1,\dotsc,y_n])
\\
& = y_n^{-|a|} \cdot \frac{ y^{\rho,A} }{ (2 \pi i)^{n-1} }
\int_{s_1,\dotsc, s_{n-1}} U_{n,a}(s_{n-1}-|a|,\dotsc,s_{1}-|a|)
\biggl[ \prod_{j=1}^{n-1} y_{j}^{-s_j} \biggr] ds_1 \cdots ds_{n-1}
\\
& = W_{n,-a}^{A}(\alpha[y_1,\dotsc,y_n]).
\end{align*}
Thus we complete our proof.
\epf


Our recursive relation for $ V_{n,b} $ is given by:

\begin{prop} \cite[Theorem 4.2]{I08} \label{MBSO}
We have
\begin{align*}
V_{n,b}(s_1,\dotsc,s_{n})
& = \frac{2^{-1}}{(4\pi i)^{2n-2}}
\int_{\scriptstyle w_1,\dotsc, w_{n-1} \atop \scriptstyle z_1,\dotsc,z_{n-1} }
V_{n-1,\widetilde{b}} (z_1,\dotsc,z_{n-1})
\\
& \times
\biggl[ \prod_{j=1}^{n-1} \Gamma_{\R}(s_j-w_j) \Gamma_{\R}(s_j-w_{j-1}-b_n)
\Gamma_{\R}(w_j-z_j) \Gamma_{\R}(w_j-z_{j-1}+b_n) \biggr]
\\
& \times
\Gamma_{\R}(s_n-w_{n-1}-b_n) \Gamma_{\R}(s_n-z_{n-1}+b_n)
\,dz_1 \cdots dz_{n-1} \, dw_1 \cdots dw_{n-1},
\end{align*}
where we understand that $ w_0 = z_0 =0 $.
Note that $ V_{1,(b_1)}(s) = 2^{-1} \Gamma_{\R}(s+b_1) \Gamma_{\R}(s-b_1) $.
\end{prop}

\begin{rem}  Technically, the formula given in the above proposition is of a different form from that given in \cite[Theorem 4.2]{I08}.  To show equality of these two formulas, one simply integrates  with respect to
$ w_1, \dotsc, w_{n-1} $ in the above formula, using   {\it Barnes' first lemma} (cf. \cite{Ba1}):
\begin{align} \label{Barnes1st}
\begin{split}
& \frac{1}{4\pi i}\int_z
\Gamma_{\R}(z+a) \Gamma_{\R}(z+b) \Gamma_{\R}(-z+c) \Gamma_{\R}(-z+d) \,dz
\\
& = \frac{\Gamma_{\R}(a+c) \Gamma_{\R}(a+d) \Gamma_{\R}(b+c) \Gamma_{\R}(b+d)}
{\Gamma_{\R} (a+b+c+d) }.
\end{split}
\end{align}
\end{rem}

\subsection{Some identities for $ U_{n,a} $ and $ V_{n,b} $}
In this subsection, we present some useful
formulas expressing $ U_{n,a} $ (respectively, $ V_{n,b}) $ as a {\it Barnes-type} integral of $ U_{n,a} $
(respectively $ V_{n,b} $).  Here, by a Barne-type integral, we mean one of Mellin-Barnes type, as described above, but lacking the ``product of powers of the independent variables.''  That is, a Barnes-type integral is an integral of a rational expression in Gamma functions.
In our discussions of such Barnes-type integrals, Barnes' first lemma (\ref{Barnes1st})
will play a fundamental role.  

We start with an easy consequence of Barnes' first lemma.

\begin{lem} \label{lemBarnes}
Let $ \gamma $, $ \delta $,
$ c_j $ and $d_j $ $ (1 \le j \le n) $ be complex numbers; also define
$ c_0 = d_0 = 0 $.
\begin{itemize}
\item[(1)] We have
\begin{align*}
\begin{split}
& \frac{1}{(4\pi i)^{n}} \int_{z_1,\dotsc,z_{n}}
\biggl[ \prod_{j=1}^{n} \Gamma_{\R} (c_{j-1}+\gamma+z_j) \Gamma_{\R}(c_j + z_j)
\Gamma_{\R}( d_{j-1}+ \delta-z_j) \Gamma_{\R}(d_j-z_j) \, dz_j\biggr]
\\
& =
\frac{ \Gamma_{\R}(c_1+\delta) \Gamma_{\R}(d_1+\gamma)
\Gamma_{\R}(\gamma+\delta) \Gamma_{\R}(c_{n}+d_{n})}
{ \Gamma_{\R}(\gamma+\delta+c_{n}+d_{n}) }
\cdot \frac{1}{(4\pi i)^{n-1}} \int_{z_1,\dotsc,z_{n-1}}
\\
& \times
\biggl[ \prod_{j=1}^{n-1}
\Gamma_{\R}(c_j+z_j) \Gamma_{\R}(c_{j+1}+\delta+z_j)
\Gamma_{\R}(d_j-z_j) \Gamma_{\R}(d_{j+1}+\gamma-z_j)
\,dz_j \biggr].
\end{split}
\end{align*}
\item[(2)] We have
\begin{align*}
\begin{split}
& \frac{1}{(4\pi i)^{n-1}} \int_{z_1,\dotsc,z_{n-1}}
\biggl[ \prod_{j=1}^{n-1} \Gamma_{\R}(c_j + z_j)
\Gamma_{\R} (c_{j+1}+\gamma+z_j) \biggr]
\\
& \times \biggl[ \prod_{j=1}^{n-1}
\Gamma_{\R}( d_{j-1}+ \delta-z_j) \Gamma_{\R}(d_j-z_j) \, dz_j\biggr]
\\
& =
\frac{ \Gamma_{\R}(c_1+\delta) \Gamma_{\R}(c_n+d_{n-1}+\gamma) }
{ \Gamma_{\R}(c_1+\gamma) \Gamma_{\R}(c_n+d_{n-1}+\delta) }
\cdot \frac{1}{(4\pi i)^{n-1}} \int_{z_1,\dotsc,z_{n-1}}
\\
& \times
\biggl[ \prod_{j=1}^{n-1}
\Gamma_{\R}(c_j+z_j) \Gamma_{\R}(c_{j+1}+\delta+z_j)
\Gamma_{\R}(d_{j-1}+\gamma-z_j) \Gamma_{\R}(d_j-z_j)
\,dz_j \biggr].
\end{split}
\end{align*}
\end{itemize}
\end{lem}

\bpf
Let first us show (1).
Using Barnes' first lemma (\ref{Barnes1st}),
we can integrate in all variables on either side.  On the left hand side, we get
\begin{align*}
& \prod_{j=1}^{n-1}
\frac{\Gamma_{\R}(c_{j-1}+d_{j-1}+\gamma+\delta) \Gamma_{\R}(c_{j-1}+d_j+\gamma)
\Gamma_{\R}(c_j+d_{j-1}+\delta) \Gamma_{\R}(c_j+d_j)}
{\Gamma_{\R}(c_{j-1}+c_j+d_{j-1}+d_j+\gamma+\delta)},
\end{align*}while on the right hand side, we get
\begin{align*}& \frac{ \Gamma_{\R}(c_1+\delta) \Gamma_{\R}(d_1+\gamma)
\Gamma_{\R}(\gamma+\delta) \Gamma_{\R}(c_{n-1}+d_{n-1})}
{\Gamma_{\R}(c_{n-1}+d_{n-1}+\gamma+\delta)}
\\
& \times
\prod_{j=1}^{n-2}
\frac{ \Gamma_{\R}(c_j+d_j) \Gamma_{\R}(c_j+d_{j+1}+\gamma)
\Gamma_{\R}(c_{j+1}+d_j+ \delta) \Gamma_{\R}(c_{j+1}+d_{j+1}+\gamma+\delta) }
{ \Gamma_{\R}(c_j+c_{j+1}+d_j+d_{j+1}+\gamma+\delta) }.
\end{align*}The two sides are then readily seen to be equal.

The latter claim (2) follows similarly from Barnes' first lemma.
\epf

\medskip

We now present some expressions, to be of use to us in the next two sections, for $U_{n,a}$ as integrals that themselves involve $U_{n,a}$.

\begin{prop} \label{lemGL}
\begin{itemize}
\item[(1)]
For a complex number $ \sigma $, we have
\begin{align} \label{lemGL1}
\begin{split}
U_{n,a}(p_1,\dotsc, p_{n-1})
& = \frac{\Gamma_{\R}(p_{n-1}+\sigma+|a|)}
{ \prod_{j=1}^{n} \Gamma_{\R}(\sigma+a_j) }
\cdot \frac{1}{(4\pi i)^{n-1}}
\int_{q_1,\dotsc,q_{n-1}}
U_{n,a}(q_1,\dotsc, q_{n-1})
\\
& \times
\biggl[ \prod_{j=1}^{n-1} \Gamma_{\R}(p_j-q_j) \Gamma_{\R}(p_{j-1}-q_{j}+\sigma) \biggr]
dq_1\cdots dq_{n-1}.
\end{split}
\end{align}
Here we understand that $ p_0 = 0 $.
\item[(2)]
For a complex number $ \sigma $, we have
\begin{align} \label{lemGL2}
\begin{split}
U_{n,a}(p_1,\dotsc, p_{n-1})
& = \frac{\Gamma_{\R}(p_{1}+\sigma)}
{ \prod_{j=1}^{n} \Gamma_{\R}(\sigma-a_j) }
\cdot \frac{1}{(4\pi i)^{n-1}}
\int_{q_1,\dotsc,q_{n-1}}
U_{n,a}(q_1,\dotsc, q_{n-1})
\\
& \times
\biggl[ \prod_{j=1}^{n-1} \Gamma_{\R}(p_j-q_j) \Gamma_{\R}(p_{j+1}-q_{j}+\sigma) \biggr]
dq_1\cdots dq_{n-1}.
\end{split}
\end{align}
Here we understand that $ p_n = -|a| $.
\end{itemize}
\end{prop}
\bpf
Our proof proceeds by induction on $n$.  The case $n=2$ amounts to Barnes' first lemma (\ref{Barnes1st}).

Now by Proposition \ref{MBGL}, the right hand side of (\ref{lemGL1}) is seen to be equal to 
\begin{align}\begin{split}& \frac{\Gamma_{\R}(p_{n-1}+\sigma+|a|)}
{ \prod_{j=1}^{n} \Gamma_{\R}(\sigma+a_j) }
\cdot \frac{2^{-1} }{(4\pi i)^{2n-3}}
\int_{ \scriptstyle q_1,\dotsc, q_{n-1} \atop
\scriptstyle r_1,\dotsc, r_{n-2} }
U_{n-1, \widetilde{a}}(r_1,\dotsc, r_{n-2})
\\
& \times
\biggl[ \prod_{j=1}^{n-1}
\Gamma_{\R} \Bigl( q_j-r_j- \frac{j a_1}{n-1} \Bigr)
\Gamma_{\R} \Bigl( q_j-r_{j-1}+ \frac{(n-j)a_1}{n-1} \Bigr) \biggr]
\\
& \times
\biggl[\prod_{j=1}^{n-1}
\Gamma_{\R}(p_j-q_j) \Gamma_{\R}(p_{j-1}-q_j+\sigma) \biggr]
\,dr_1 \cdots dr_{n-2} \, dq_1 \cdots dq_{n-1},\label{rtside1}
\end{split}\end{align}
where $r_0=0$ and $ r_{n-1}  = -|a| $.
To the integral in the $q_j$'s in (\ref{rtside1}), we now apply Lemma \ref{lemBarnes} (1) with $ \gamma = a_1 $; $ \delta = \sigma $; and, for $1\le j\le n-1$, 
$ c_j = -r_j - ja_1/(n-1) $ and $ d_j = p_j $.  We thereby find that (\ref{rtside1}) equals
\begin{align}\begin{split}&\ \frac{\Gamma_{\R}(p_{n-1}+\sigma+|a|)}
{ \prod_{j=1}^{n} \Gamma_{\R}(\sigma+a_j) }
\cdot \frac{2^{-1}}{(4\pi i)^{2n-4}}
\int_{ \scriptstyle q_1,\dotsc, q_{n-2} \atop
\scriptstyle r_1,\dotsc, r_{n-2} }
U_{n-1, \widetilde{a}}(r_1,\dotsc, r_{n-2})
\\
& \times
\frac{ \Gamma_{\R}(-r_1-\frac{a_1}{n-1} + \sigma) \Gamma_{\R}(p_1+a_1)
\Gamma_{\R}(\sigma+a_1) \Gamma_{\R}(p_{n-1}-a_1+|a|) }
{ \Gamma_{\R}(p_{n-1}+\sigma+a_1+\cdots+a_n) }
\\
& \times
\biggl[ \prod_{j=1}^{n-2}
\Gamma_{\R} \Bigl(q_j-r_j-\frac{ja_1}{n-1} \Bigr)
\Gamma_{\R} \Bigl(q_j-r_{j+1}+\sigma-\frac{(j+1)a_1}{n-1} \Bigr) \biggr]
\\
& \times
\biggl[ \prod_{j=1}^{n-2}
\Gamma_{\R}( p_j-q_j) \Gamma_{\R}(p_{j+1}-q_j+a_1) \biggr]
\,dr_1 \cdots dr_{n-2} \, dq_1 \cdots dq_{n-2}
\\
& = \frac{1}{ \prod_{j=2}^n \Gamma_{\R}(\sigma + a_j) }
\cdot \frac{2^{-1}}{(4 \pi i)^{2n-4}}
\int_{ \scriptstyle q_1,\dotsc, q_{n-2} \atop
\scriptstyle r_1,\dotsc, r_{n-2} }
U_{n-1, \widetilde{a}}(r_1,\dotsc, r_{n-2})
\\
& \times \Gamma_{\R}(p_{n-1}-a_1+|a|)
\Gamma_{\R}(q_{n-2}+\sigma-a_1+|a|)
\Gamma_{\R}(p_{n-1}-q_{n-2}+a_1)
\\
& \times
\biggl[ \prod_{j=1}^{n-2}
\Gamma_{\R} \Bigl( q_j-r_j-\frac{ja_1}{n-1} \Bigr)
\Gamma_{\R} \Bigl( q_{j-1}-r_j-\frac{ja_1}{n-1} + \sigma \Bigr) \biggr]
\\
& \times \biggl[ \prod_{j=1}^{n-2}
\Gamma_{\R}( p_j-q_j) \Gamma_{\R}(p_j-q_{j-1}+a_1) \biggr]
\,dr_1\cdots dr_{n-2} \,dq_1 \cdots dq_{n-2}.\label{rtside2}
\end{split}\end{align}
By the  induction hypothesis, we can integrate with respect to
the $ r_j $'s, to find that (\ref{rtside2}) equals
\begin{align}\begin{split}\label{rtside3}
& \frac{1}{\prod_{j=2}^n \Gamma_{\R}(\sigma+a_j)}
\cdot \frac{2^{-1}}{(4\pi i)^{n-2}} \int_{q_1,\dotsc, q_{n-2}}
\frac{ \prod_{j=1}^{n-1} \Gamma_{\R}(\sigma-\frac{a_1}{n-1}+\widetilde{a}_j) }
{ \Gamma_{\R}(q_{n-2} + \sigma - a_1 + |\widetilde{a}| ) }
\\
& \times U_{n,\widetilde{a}}
\Bigl(q_1-\frac{a_1}{n-1}, \dotsc, q_{n-2}-\frac{(n-2)a_1}{n-1} \Bigr)
\\
& \times \Gamma_{\R}(p_{n-1}-a_1+|a|) \Gamma_{\R}(p_{n-1}-q_{n-2}+a_1)
\Gamma_{\R}(q_{n-2}+\sigma-a_1+|a|)
\\
& \times
\biggl[ \prod_{j=1}^{n-2} \Gamma_{\R}(p_j-q_j) \Gamma_{\R}(p_j-q_{j-1}+a_1) \biggr]
\, dq_1 \cdots dq_{n-2}
\\
& = \frac{2^{-1}}{(4\pi i)^{n-2}}
\int_{q_1,\dotsc,q_{n-2}} U_{n,\widetilde{a}}
\Bigl(q_1-\frac{a_1}{n-1}, \dotsc, q_{n-2}-\frac{(n-2)a_1}{n-1} \Bigr)
\\
& \times \biggl[ \prod_{j=1}^{n-1}
\Gamma_{\R}(p_j-q_j) \Gamma_{\R}(p_j-q_{j-1}+a_1) \biggr] \,dq_1 \cdots dq_{n-2}.\end{split}
\end{align}
Here we've used the fact that  that $ |a| = |\widetilde{a}| $; also, we've defined $ q_{n-1} = a_1-|a| $.

Finally, after the substitution $ q_j \to q_j+ja_1/(n-1) $ for $ 1 \le j \le n-2 $,
we use Proposition \ref{MBGL} to find that (\ref{rtside3})
becomes $ U_{n,a}(p_1,\dotsc,p_{n-1}) $.
Thus we complete the proof of (\ref{lemGL1}).
\medskip

The proof of  (\ref{lemGL2}) is similar:  Again, we use induction on $n$, with the case $n=2$ being equivalent to Barnes' first lemma (\ref{Barnes1st}).

To rearrange the integration over the $ q_j $'s,
we first rewrite the identity in Lemma \ref{lemBarnes} (2) slightly,
by integrating with respect to $z_{n-1}$ on the right hand side.   We get
\begin{align*}
\begin{split}
& \frac{1}{(4\pi i)^{n-1}} \int_{z_1,\dotsc,z_{n-1}}
\biggl[ \prod_{j=1}^{n-1} \Gamma_{\R}(c_j+z_j) \Gamma_{\R}(c_{j+1}+\gamma+z_j)
\Gamma_{\R}(d_{j-1}+\delta-z_j) \Gamma_{\R}(d_j-z_j) \,dz_j \biggr]
\\
& = \frac{\Gamma_{\R}(c_1+\delta) \Gamma_{\R}(c_n+d_{n-1}+\gamma)
\Gamma_{\R}(c_{n-1}+d_{n-2}+\gamma) }
{\Gamma_{\R}(c_1+\gamma) \Gamma_{\R}(c_{n-1}+c_n+d_{n-2}+d_{n-1}+\gamma+\delta)}
\\
& \times \Gamma_{\R}(c_{n-1}+d_{n-1}) \Gamma_{\R}(c_n+d_{n-2}+\gamma+\delta)
\\
& \times
\frac{1}{(4\pi i)^{n-2}} \int_{z_1,\dotsc,z_{n-2}}
\biggl[ \prod_{j=1}^{n-2} \Gamma_{\R}(c_j+z_j) \Gamma_{\R}(c_{j+1}+\delta+z_j) \biggr]
\\
& \times
\biggl[ \prod_{j=1}^{n-2} \Gamma_{\R}(d_{j-1}+\gamma-z_j) \Gamma_{\R}(d_j-z_j) \,dz_j \biggr].
\end{split}
\end{align*}
Applying this formula with
$ c_j = p_j $ $ (1 \le j \le n-1) $, $ c_n = -|a| $, $ d_j = -r_j-ja_1/(n-1) $
$ (1 \le j \le n-2) $, $ d_{n-1} = -a_1+|a| $, $ \gamma = \sigma $ and
$ \delta = a_1$, we find that 
the right hand side of (\ref{lemGL2}) becomes
\begin{align*}
& \frac{ 1}{ \prod_{j=2}^n \Gamma_{\R}(\sigma-a_j) } \cdot \frac{2^{-1} }{(4\pi i)^{2n-4}}
\int_{ \scriptstyle q_1,\dotsc, q_{n-2} \atop \scriptstyle r_1,\dotsc,r_{n-2}}
U_{n-1,\widetilde{a}}(r_1, \dotsc, r_{n-2})
\\
& \times
\Gamma_{\R}(p_1+a_1) \Gamma_{\R}(p_{n-1}-a_1+|a|)
\biggl[ \prod_{j=1}^{n-2} \Gamma_{\R}(p_j-q_j) \Gamma_{\R}(p_{j+1}-q_j+a_1) \biggr]
\\
& \times
\Gamma_{\R}(q_1+\sigma)
\biggl[ \prod_{j=1}^{n-2} \Gamma_{\R}\Bigl(q_{j+1}-r_j-\frac{ja_1}{n-1} + \sigma \Bigr)
\Gamma_{\R} \Bigl(q_j-r_j-\frac{ja_1}{n-1} \Bigr) \biggr]
\\
& \times dr_1 \cdots dr_{n-2} \, dq_1 \cdots dq_{n-2}.
\end{align*}
Using the   induction hypothesis to integrate in the  $r_j$'s,
together with Proposition \ref{MBGL}, we prove our assertion.
\epf
\medskip

Our $SO_{2n+1}(\R)$-analog of Proposition \ref{lemGL} is as follows.

\begin{prop} \label{lemSO}
For a complex number $ \sigma $, we have
\begin{align} \label{lemSOeq}
\begin{split}
& V_{n,b}(p_1,\dotsc,p_n)
\\
& = \frac{1}{\prod_{j=1}^n \Gamma_{\R}(\sigma+b_j) \Gamma_{\R}(\sigma-b_j) }
\cdot \frac{1}{(4\pi i)^{2n-1}}
\int_{\scriptstyle q_1,\dotsc,q_{n-1} \atop \scriptstyle r_1,\dotsc,r_{n}}
V_{n,b}(r_1,\dotsc,r_{n})
\\
& \times
\biggl[ \prod_{j=1}^{n-1} \Gamma_{\R}(p_j-q_j)
\Gamma_{\R}(q_j-r_j) \biggr]
\biggl[ \prod_{j=1}^n
\Gamma_{\R}(p_j-q_{j-1}+\sigma) \Gamma_{\R}(q_{j-1}-r_{j}+\sigma) \biggr]
\\
& \times
\Gamma_{\R}(p_n-r_n)\,dr_1 \cdots dr_{n} \, dq_1 \cdots dq_{n-1}.
\end{split}
\end{align}
Here we understand that $q_0=0$.
\end{prop}

\bpf
We again proceed by induction  on $n$.  The case $n=1$ is an immediate consequence of Barnes' first lemma (\ref{Barnes1st}).

We now apply Proposition \ref{MBSO}, to find that the right hand side of (\ref{lemSOeq}) equals
\begin{align} \label{RHSint1}
\begin{split}
& \frac{1}{\prod_{j=1}^n \Gamma_{\R}(\sigma+b_j) \Gamma_{\R}(\sigma-b_j) }\cdot\frac{2^{-1} }{(4\pi i)^{4n-3}}
\int_{\scriptstyle q_1,\dotsc,q_{n-1} \atop \scriptstyle r_1,\dotsc,r_{n}}
\int_{\scriptstyle w_1,\dotsc,w_{n-1} \atop \scriptstyle z_1,\dotsc,z_{n-1}}
V_{n-1,\widetilde{b}}(z_1,\dotsc,z_{n-1})
\\
& \times
\biggl[ \prod_{j=1}^{n-1} \Gamma_{\R}(r_j-w_j) \Gamma_{\R}(r_j-w_{j-1}-b_n)
\Gamma_{\R}(w_j-z_j) \Gamma_{\R}(w_j-z_{j-1}+b_n) \biggr]
\\
& \times
\Gamma_{\R}(r_n-z_{n-1}+b_n) \Gamma_{\R}(r_n-w_{n-1}-b_n)
\\
& \times
\biggl[ \prod_{j=1}^{n-1} \Gamma_{\R}(p_j-q_j) \Gamma_{\R}(q_j-r_j) \biggr]
\biggl[ \prod_{j=1}^n \Gamma_{\R}(p_j-q_{j-1}+\sigma)\Gamma_{\R}(q_{j-1}-r_{j}+\sigma) \biggr]
\\
& \times
\Gamma_{\R}(p_n-r_n) \,
dz_1 \cdots dz_{n-1} \, dw_1 \cdots dw_{n-1} \,
dr_1 \cdots dr_{n} \, dq_1 \cdots dq_{n-1}.
\end{split}
\end{align}
To the integral in the   $r_j$'s, in  (\ref{RHSint1}), we now apply   Lemma \ref{lemBarnes} (1), which yields:
\begin{align*}
& \frac{1}{(4\pi i)^{n}} \int_{r_1,\dotsc,r_n}
\biggl[\prod_{j=1}^{n-1} \Gamma_{\R}(r_j-w_j) \Gamma_{\R}(r_j-w_{j-1}-b_n)
\Gamma_{\R}(q_j-r_j) \Gamma_{\R}(q_{j-1}-r_{j}+\sigma) \biggr]
\\
& \times
\Gamma_{\R}(r_n-z_{n-1}+b_n) \Gamma_{\R}(r_n-w_{n-1}-b_n)
\Gamma_{\R}(p_n-r_n) \Gamma_{\R}(q_{n-1}-r_n+\sigma)
\,dr_1 \cdots dr_n
\\
& = \frac{ \Gamma_{\R} (\sigma-w_1) \Gamma_{\R}(q_1-b_n) \Gamma_{\R}(\sigma-b_n)
\Gamma_{\R}(p_n-z_{n-1}+b_n) }{\Gamma_{\R}(p_n-z_{n-1}+\sigma) }
\cdot \frac{1}{(4\pi i)^{n-1}} \int_{r_1,\dotsc,r_{n-1}}
\\
& \times
\biggl[ \prod_{j=1}^{n-2} \Gamma_{\R}(r_j-w_j) \Gamma_{\R}(r_j-w_{j+1}+\sigma)
\Gamma_{\R}(q_j-r_j) \Gamma_{\R}(q_{j+1}-r_j-b_n) \biggr]
\\
& \times
\Gamma_{\R}(r_{n-1}-w_{n-1}) \Gamma_{\R}(r_{n-1}-z_{n-1}+\sigma+b_n)
\Gamma_{\R}(q_{n-1}-r_{n-1}) \Gamma_{\R}(p_n-r_{n-1}-b_n)
\\
& \times
dr_1 \cdots dr_{n-1}
\\
& = \frac{ \Gamma_{\R}(\sigma-b_n) \Gamma_{\R}(p_n-z_{n-1}+b_n) }
{ \Gamma_{\R}(p_n-z_{n-1}+\sigma) }
\cdot \frac{1}{(4\pi i)^{n-1}} \int_{r_1,\dotsc,r_{n-1}}
\\
& \times
\biggl[ \prod_{j=1}^{n-1} \Gamma_{\R}(r_j-w_j) \Gamma_{\R}(r_{j-1}-w_{j}+\sigma)
\Gamma_{\R}(q_j-r_j) \Gamma_{\R}(q_{j}-r_{j-1}-b_n) \biggr]
\\
& \times \Gamma_{\R}(r_{n-1}-z_{n-1}+\sigma+b_n) \Gamma_{\R}(p_n - r_{n-1}-b_n)
\, dr_1 \cdots dr_{n-1}.
\end{align*}
Then (\ref{RHSint1}) equals
\begin{align}   
\begin{split}\label{RHSint2}
&\frac{1}{\prod_{j=1}^n \Gamma_{\R}(\sigma+b_j) \Gamma_{\R}(\sigma-b_j) }\cdot \frac{2^{-1} }{(4\pi i)^{4n-4}}
\int_{\scriptstyle q_1,\dotsc,q_{n-1} \atop \scriptstyle r_1,\dotsc,r_{n-1}}
\int_{\scriptstyle w_1,\dotsc,w_{n-1} \atop \scriptstyle z_1,\dotsc,z_{n-1}}
V_{n-1,\widetilde{b}}(z_1,\dotsc,z_{n-1})
\\
& \times
\biggl[ \prod_{j=1}^{n-1} \Gamma_{\R}(w_j-z_j) \Gamma_{\R}(w_j-z_{j-1}+b_n)
\Gamma_{\R}(r_j-w_j) \Gamma_{\R}(r_{j-1}-w_j+\sigma) \biggr]
\\
& \times
\Gamma_{\R}(p_1+\sigma)
\biggl[ \prod_{j=1}^{n-1} \Gamma_{\R}(p_j-q_j) \Gamma_{\R}(p_{j+1}-q_{j}+\sigma)
\Gamma_{\R}(q_j-r_j) \Gamma_{\R}(q_j-r_{j-1}-b_n) \biggr]
\\
& \times \frac{\Gamma_{\R}(\sigma-b_n) \Gamma_{\R}(p_n-z_{n-1}+b_n)
\Gamma_{\R}(r_{n-1}-z_{n-1}+\sigma+b_n) \Gamma_{\R}(p_n-r_{n-1}-b_n)}
{\Gamma_{\R}(p_n-z_{n-1}+\sigma)}
\\
& \times
dz_1 \cdots dz_{n-1} \, dw_1 \cdots dw_{n-1} \,
dr_1 \cdots dr_{n-1} \, dq_1 \cdots dq_{n-1}. 
 \end{split}
\end{align}

We apply Lemma \ref{lemBarnes} (1) with
$ (c_j,d_j,\gamma,\delta) = (-z_j,r_j,\nu_n,\sigma) $
to the integratlin the $w_j$'s, and
Lemma \ref{lemBarnes} (2) with
$ (c_j,d_j,\gamma,\delta) =( p_j ,-r_j,\sigma, -\nu) $
to the integral in the $q_j $'s.
Then (\ref{RHSint2}) becomes
\begin{align}\begin{split}
& \frac{1}{\prod_{j=1}^n \Gamma_{\R}(\sigma+b_j) \Gamma_{\R}(\sigma-b_j) }\cdot\frac{2^{-1} }{(4\pi i)^{4n-5}}
\int_{\scriptstyle q_1,\dotsc,q_{n-1} \atop \scriptstyle r_1,\dotsc,r_{n-1}}
\int_{\scriptstyle w_1,\dotsc,w_{n-2} \atop \scriptstyle z_1,\dotsc,z_{n-1}}
V_{n-1,\widetilde{b}}(z_1,\dotsc,z_{n-1})
\\
& \times
\frac{ \Gamma_{\R}(\sigma-z_1) \Gamma_{\R}(r_1+b_n)
\Gamma_{\R}(\sigma+b_n) \Gamma_{\R}(r_{n-1}-z_{n-1}) }
{\Gamma_{\R}(r_{n-1}-z_{n-1} + \sigma +b_n) }
\\
& \times
\biggl[ \prod_{j=1}^{n-2}
\Gamma_{\R}(w_j-z_j) \Gamma_{\R}(w_j-z_{j+1}+\sigma)
\Gamma_{\R}(r_j-w_j) \Gamma_{\R}(r_{j+1}-w_j+b_n) \biggr]
\\
& \times
\frac{ \Gamma_{\R}(p_1-b_n) \Gamma_{\R}(p_n-r_{n-1}+\sigma) }
{ \Gamma_{\R}(p_1+\sigma) \Gamma_{\R}(p_n-r_{n-1}-b_n) }
\\
& \times
\biggl[ \prod_{j=1}^{n-1}
\Gamma_{\R}(p_j-q_j) \Gamma_{\R}(p_{j+1}-q_j-b_n)
\Gamma_{\R}(q_j-r_j) \Gamma_{\R}(q_j-r_{j-1}+\sigma) \biggr]
\\
& \times
\frac{ \Gamma_{\R}(\sigma-b_n) \Gamma_{\R}(p_1+\sigma)
\Gamma_{\R}(p_n-z_{n-1}+b_n) \Gamma_{\R}(p_n-r_{n-1}-b_n)}
{ \Gamma_{\R}(p_n-z_{n-1}+\sigma) }
\\
& \times
\Gamma_{\R}(r_{n-1}-z_{n-1}+\sigma+b_n)
\,dz_1 \cdots dz_{n-1} \, dw_1 \cdots dw_{n-2} \,
dr_1 \cdots dr_{n-1} \, dq_1 \cdots dq_{n-1}.\label{RHSint25}\end{split}
\end{align}
We rewrite (\ref{RHSint25}) as
\begin{align} \label{RHSint3}
\begin{split}
& \frac{1}{\prod_{j=1}^{n-1} \Gamma_{\R}(\sigma+b_j) \Gamma_{\R}(\sigma-b_j) }
\cdot \frac{ 2^{-1} }{(4\pi i)^{4n-5}}
\int_{\scriptstyle q_1,\dotsc,q_{n-1} \atop \scriptstyle r_1,\dotsc,r_{n-1}}
\int_{\scriptstyle w_1,\dotsc,w_{n-2} \atop \scriptstyle z_1,\dotsc,z_{n-1}}
V_{n-1,\widetilde{b}}(z_1,\dotsc,z_{n-1})
\\
& \times
\biggl[\prod_{j=1}^{n-2} \Gamma_{\R}(r_j-w_j) \Gamma_{\R}(r_j-w_{j-1}+b_n)
\Gamma_{\R}(q_j-r_j) \Gamma_{\R}(q_{j+1}-r_j+\sigma) \biggr]
\\
& \times
\Gamma_{\R}(r_{n-1}-z_{n-1}) \Gamma_{\R}(r_{n-1}-w_{n-2}+b_n)
\Gamma_{\R}(q_{n-1}-r_{n-1}) \Gamma_{\R}(p_n-r_{n-1}+\sigma)
\\
& \times
\frac{ \Gamma_{\R}(q_1+\sigma) \Gamma_{\R}(p_n-z_{n-1}+\nu_n) }
{ \Gamma_{\R}(p_n-z_{n-1}+\sigma) }
\biggl[ \prod_{j=1}^{n-2} \Gamma_{\R}( w_j-z_j) \biggr]
\\
& \times
\biggl[\prod_{j=1}^{n-1} \Gamma_{\R}(w_{j-1}-z_j+\sigma) \biggr]
\biggl[ \prod_{j=1}^{n-1} \Gamma_{\R}(p_j-q_j) \biggr]
\biggl[ \prod_{j=1}^n \Gamma_{\R}(p_j-q_{j-1}-b_n) \biggr]
\\
& \times dz_1 \cdots dz_{n-1} \, dw_1 \cdots dw_{n-2} \,
dr_1 \cdots dr_{n-1} \, dq_1 \cdots dq_{n-1}.
\end{split}
\end{align}

We apply Lemma \ref{lemBarnes} (2) to the integral in $ r_1,\dotsc, r_{n-1} $,to find that
(\ref{RHSint3})  equals
\begin{align}\begin{split}\label{RHSint4}
&\frac{1}{\prod_{j=1}^{n-1} \Gamma_{\R}(\sigma+b_j) \Gamma_{\R}(\sigma-b_j) }\cdot
\frac{ 2^{-1} }{(4\pi i)^{4n-5}}
\int_{\scriptstyle q_1,\dotsc,q_{n-1} \atop \scriptstyle r_1,\dotsc,r_{n-1}}
\int_{\scriptstyle w_1,\dotsc,w_{n-2} \atop \scriptstyle z_1,\dotsc,z_{n-1}}
V_{n-1,\widetilde{b}}(z_1,\dotsc,z_{n-1})
\\
& \times \Gamma_{\R}(q_1+b_n)
\biggl[ \prod_{j=1}^{n-1} \Gamma_{\R}(q_j-r_j) \biggr]
\biggl[ \prod_{j=1}^{n-2} \Gamma_{\R}(q_{j+1}-r_j+b_n) \biggr]
\Gamma_{\R}(p_n-r_{n-1}+b_n)
\\
& \times
\biggl[ \prod_{j=1}^{n-2} \Gamma_{\R}(r_j-w_j) \biggr]
\Gamma_{\R}(r_{n-1}-z_{n-1})
\biggl[ \prod_{j=1}^{n-1} \Gamma_{\R}(r_j-w_{j-1}+\sigma) \biggr]
\\
& \times
\biggl[ \prod_{j=1}^{n-2} \Gamma_{\R}( w_j-z_j) \biggr]
\biggl[\prod_{j=1}^{n-1} \Gamma_{\R}(w_{j-1}-z_j+\sigma) \Gamma_{\R}(p_j-q_j) \biggr]
\\
& \times \biggl[ \prod_{j=1}^n \Gamma_{\R}(p_j-q_{j-1}-b_n) \biggr]
dz_1 \cdots dz_{n-1} \, dw_1 \cdots dw_{n-2} \,
dr_1 \cdots dr_{n-1} \, dq_1 \cdots dq_{n-1}.
\\
& = \frac{1}{\prod_{j=1}^{n-1} \Gamma_{\R}(\sigma+b_j) \Gamma_{\R}(\sigma-b_j) }\cdot
\frac{ 2^{-1} }{(4\pi i)^{4n-5}}
\int_{\scriptstyle q_1,\dotsc,q_{n-1} \atop \scriptstyle r_1,\dotsc,r_{n-1}}
\int_{\scriptstyle w_1,\dotsc,w_{n-2} \atop \scriptstyle z_1,\dotsc,z_{n-1}}
V_{n-1,\widetilde{b}}(z_1,\dotsc,z_{n-1})
\\
& \times \biggl[ \prod_{j=1}^{n-2} \Gamma_{\R}(r_j-w_j) \Gamma_{\R}(w_j-z_j) \biggr]
\\
& \times
\biggl[ \prod_{j=1}^{n-1} \Gamma_{\R}(w_{j-1}-z_j+\sigma) \Gamma_{\R}(r_j-w_{j-1}+\sigma)
\biggr]
\Gamma_{\R}(r_{n-1}-z_{n-1})
\\
& \times
\biggl[ \prod_{j=1}^{n-1}
\Gamma_{\R}(p_j-q_j) \Gamma_{\R}(p_j-q_{j-1}-b_n) \Gamma_{\R}(q_j-r_j)
\Gamma_{\R}(q_j-r_{j-1}+b_n) \biggr]
\\
& \times
\Gamma_{\R}(p_n-r_{n-1}+b_n) \Gamma_{\R}(p_n-q_{n-1}-b_n)
\\
& \times dz_1 \cdots dz_{n-1} \, dw_1 \cdots dw_{n-2} \,
dr_1 \cdots dr_{n-1} \, dq_1 \cdots dq_{n-1}.\end{split}
\end{align}
By the induction hypothesis, we can perform the integration
over the $ w_j $'s and $ z_j $'s.
Thus (\ref{RHSint4}) becomes
\begin{align*}
&  \frac{ 2^{-1} }{(4 \pi i)^{2n-2}}
\int_{\scriptstyle q_1,\dotsc,q_{n-1} \atop \scriptstyle r_1,\dotsc,r_{n-1}}
V_{n-1,\widetilde{b}}(r_1,\dotsc,r_{n-1})
\\
& \times
\biggl[ \prod_{j=1}^{n-1}
\Gamma_{\R}(p_j-q_j) \Gamma_{\R}(p_j-q_{j-1}-b_n) \Gamma_{\R}(q_j-r_j)
\Gamma_{\R}(q_j-r_{j-1}+b_n) \biggr]
\\
& \times
\Gamma_{\R}(p_n-r_{n-1}+b_n) \Gamma_{\R}(p_n-q_{n-1}-b_n)
\, dr_1 \cdots dr_{n-1} \, dq_1 \cdots dq_{n-1}
\\
& =  V_{n,b}(p_1,\dotsc,p_{n-1}).
\end{align*}
The last equality follows from
Proposition \ref{MBSO}, and   our proof is complete.
\epf

\section{archimedean zeta integrals on $GL_{n} \times GL_m $}

In this section, we explicitly calculate archimedean zeta integrals for degree $nm$ $L$-functions on
$GL_n \times GL_m $ when $ 0\le n-m \le 2$,
by applying some formulas given in the previous section.
When $ m = n$ or $n-1$, and in the case when the class one principal series representations in question are induced from characters trivial on the center of $GL_n$ and $GL_m$, the  second author \cite{St01} \cite{St02} has proved the coincidence of the
local zeta integrals and the local Langlands $L$-factors,
by way of a recursive relation between $ U_{n,a} $ and $ U_{n-2,a} $ obtained in \cite{St01}.
Here we give another proof, which uses a relation between $ U_{n,a} $ and $ U_{n-1,a} $,
as well as  Proposition \ref{lemGL} in \S 2.2 below.  The present proof will not require the assumption of trivial central characters that was stipulated in the earlier proof.

For general $n$ and $m$, Jacquet and Shalika \cite{JS3} have proved local functional equations for,
and non-vanishing of, the ratio of the archimedean zeta integral to the corresponding $L$-factor. If $ n>m+1 $,  it is not generally expected that the archimedean zeta integrals
should coincide with the local $L$-factors.   And indeed Hoffstein and Murty \cite{HM},   in the case of $ (n,m) = (3,1) $, have expressed the ratio of this local integral to this local $L$-function as a certain integral of Barnes type.

In this section we generalize the work of Hoffstein and Murty, by expressing this ratio of  archimedean zeta integral to archimedean $L$-factor, in the case of $ GL_n \times GL_{n-2}$ for any $n$, explicitly as a Barnes-type integral.

\subsection{Barnes integral expressions for archimedean zeta integrals}

We first recall the archimedean zeta integrals for the standard $L$-functions on $GL_n \times GL_m $.
Let $ W $ and $ W' $ be $ \psi $ and $ \psi^{-1} $-Whittaker functions
on $GL_n(\R) $ and $ GL_m(\R) $, respectively.
Let $ X_n(\R) $ be the standard maximal unipotent subgroup
of $GL_n(\R) $ consisting of upper triangular unipotent matrices.
The archimedean zeta integrals we want to study are
\begin{align*}
Z(s, W, W') &=\begin{cases}\ds\int_{X_n(\R) \backslash GL_n(\R)}
W(g) W'(g) \Phi( e_n g) \, |\det g|^s \,dg&\mbox{ if } m=n,\\
\ds\int_{X_m(\R) \backslash GL_m(\R)}W \begin{pmatrix} g & \\ & 1_{n-m} \end{pmatrix}
 \,
W'(g) \, |\det g|^{s- {(n-m)}/{2}} \,dg &\mbox{ if } m<n.\end{cases}\end{align*}
Here $ \Phi $ is a Schwartz-Bruhat function on $\R^n$ defined by
$$
\Phi(x) = \exp \Bigl(-\pi \sum_{i=1}^n x_i^2 \Bigr)\quad (
x = (x_1,\dotsc, x_n ) \in \R^n), $$
and $ e_n = (0,\dotsc,0,1) \in \R^n $.

In the case $m\le n-2$, the  zeta integral\begin{align*}
 &Z^\vee  (s,W,W') \\& =
\int_{M_{n-m-1,m}(\R)} \int_{X_m(\R) \backslash GL_m(\R)}
 W\begin{pmatrix} g & & \\ x & 1_{n-m-1} & \\ & & 1 \end{pmatrix}\, W'(g)\, |\det g|^{s- {(n-m)}/{2}} \,dx\,dg\end{align*}
will also be of interest.

Now let us assume that $ W $ and $ W' $ are class one Whittaker functions. We write $W=W^A_{n,a}$ and $W'=\overline{W}^A_{m,a'}$, where $a=(a_1,a_2,\ldots, a_n)\in \C^n$, $a'=(a_1',a_2',\ldots, a_m')\in \C^m$, and $ \overline{W}_{m,a'}^A $ denotes the $\psi^{-1}$-Whittaker function
determined by
$ \overline{W}_{m,a'}^A (g) = \psi^{-1} (x(g)) W_{m,a'}^A(y(g)) $. 
We then define

\begin{align*}  I_{n,m;a,a'}(s)  =  Z(s, W_{n,a}^A, \overline{W}_{m,a'}^A)\end{align*} 
and  \begin{align*} 
{I}^\vee_{n,n-2;a,a'}(s)  =
Z^{\vee}(s,   (W_{n,a}^A)^{\vee},(\overline{W}_{n-2,a'}^A)^{\vee}).
 \end{align*} The  goals cited at the outset of this  section may now be described more explicitly, as follows:

\medskip\noindent{\bf 1.}  To show that $I_{n,m;a,a'}(s) $ equals the corresponding  local $L$-factor, which is a product of $(n-j)n$ Gamma functions, in the cases $m=n-j$ for $j=0$ or $1$;

\medskip\noindent{\bf 2.}  To show that $I_{n,n-2;a,a'}(s) $ equals the corresponding  local $L$-factor, which is a product of $(n-2)n$ Gamma functions, {\it times} a certain Barnes-type integral;   that $ I^\vee_{n-2,n;a,a'}(s) $ similarly equals an $L$-factor times a Barnes-type integral; and that the former Barnes-type integral becomes the latter under the replacement of $s$ by $1-s$.

The first of these goals will be accomplished in subsection 2.2 below -- see Theorem 2.1 -- and the second in subsection 2.3 -- see Theorems \ref{thmnminustwo} and \ref{GLunip}.  In the present subsection, we develop some Barnes-type integral expressions for $I_{n,m;a,a'}(s)$ and $ I^\vee_{n,n-2;a,a'}(s) $, to be of use in what follows.

\begin{prop}\label{Barnesforms} \begin{itemize}\item[(1)] We have
  \begin{align*}
I_{n,n;a,a'}(s)
& = \Gamma_{\R}(ns+|a|+|a'|)
\\
& \times
\frac{ 2^{n-1} }{ (2\pi i)^{n-1}}
\int_{\scriptstyle s_1,\dotsc,s_{n-1}}
U_{n,a}(s-s_1,2s-s_2,\dotsc, (n-1)s-s_{n-1})
\\
& \times U_{n,a'}(s_1,\dotsc,s_{n-1}) \,ds_1 \cdots ds_{n-1}.
\end{align*}

\item[(2)] For $m<n$, 
\begin{align*}
\begin{split}
 I_{n,m;a,a'}(s)  &= \frac{2^m}{(2\pi i)^{n-2}} 
   \int_{\scriptstyle s_1,\dotsc,s_{m-1} \atop \scriptstyle s_{m+1},\dotsc,s_{n-1}} 
 U_{n,a}(s-s_1, 2s-s_2, \dotsc, (m-1)s-s_{m-1},
\\
 & \qquad \qquad \qquad \qquad   ms+|a'| ,s_{m+1},\dotsc,s_{n-1})
\\
& \times 
 U_{m,a'}(s_1,\dotsc,s_{m-1}) 
 \,ds_1 \cdots ds_{m-1} ds_{m+1}\cdots ds_{n-1}.
\end{split}
\end{align*}

\item[(3)] We have  \begin{align*} 
{I}^\vee_{n,n-2;a,a'}(s)  =  \frac{1}{4\pi i} \int_{w} \frac{ I_{ n ,n-1,-a,a^\vee}\biggl(\ds\frac{3}{2}+  \frac{w+|a'|-1}{n-1}  \biggr)}{ \prod_{j=2}^{n-1}
\Gamma_{\R}(a_1^\vee-a_j^\vee+1) } \,dw,\end{align*} where
$ a^{\vee} = (a_1^{\vee},\dotsc, a_{n-1}^{\vee}) $ with
\begin{align}\begin{split}
a_1^{\vee} & = (n-2)\Bigl(-s+\frac{3}{2} + \frac{w+|a'|-1}{n-1} \Bigr);
\\
a_j^{\vee} & = -a_{j-1}' + s-\frac{3}{2} - \frac{w+|a'|-1}{n-1} \ \ \ (2 \le j \le n-1).\label{acheck}\end{split}
\end{align}

 \end{itemize}\end{prop}  

\begin{proof} We first consider part (1).   Since our Whittaker functions $ W_{n,a}^A $ and
$ \overline{W} _{n,a'}^A $ are
right $ O(n) $  invariant,   the Iwasawa decomposition implies that
\begin{align*}
I_{n,n;a,a'}(s) &  = Z(s, W_{n,a}^A, \overline{W}_{n,a'}^A)\\&= 2^n \int_{(\R_+)^n}
 \exp(-\pi y_n^2)
\\
& \times  {W}_{n,a}^A(\alpha[y_1,\dotsc, y_{n}])
 {W}_{n,a'}^A(\alpha[y_1,\dotsc, y_{n}])
\biggl[ \prod_{j=1}^n y_j^{js-j(n-j)} \frac{dy_j}{y_j} \biggr]\\&=
 2^n \int_{(\R_+)^n}
y_n^{|a| + |a'| } \exp(-\pi y_n^2)
\\
& \times \widehat{W}_{n,a}^A(\alpha[y_1,\dotsc, y_{n-1},1])
\widehat{W}_{n,a'}^A(\alpha[y_1,\dotsc, y_{n-1},1])
\biggl[ \prod_{j=1}^n y_j^{js} \frac{dy_j}{y_j} \biggr].
\end{align*}
The integral in  $ y_n $ equals
$ 2^{-1} \Gamma_{\R}(ns+|a|+|a'|) $. Applying 
  Mellin inversion and the ``convolution theorem'' -- which says that the Mellin transform of a product equals the convolution of the corresponding Mellin transforms -- to the integral in $y_1,y_2,\ldots, y_{n-1}$, above, then yields the stated formula for $I_{n,n,a,a'}(s)$.

We now prove part (2). In this case, the Iwasawa decomposition  gives us
\begin{align}\begin{split}\label{Inem}
& I_{n,m;a,a'}(s)= Z(s, W_{n,a}^A, \overline{W}_{m,a'}^A)
\\
&
= 2^m \int_{(\R_+)^m} W_{n,a}^A(\alpha[y_1,\dotsc,y_m,1,\dotsc,1]) W_{m,a'}^A(\alpha[y_1,\dotsc, y_m])
\biggl[ \prod_{j=1}^{m} y_j^{j(s- {(n-m)}/{2})-j(m-j)}
\frac{dy_j}{y_j} \biggr]
\\
& = 2^m \int_{(\R_+)^m}
\widehat{W}_{n,a}^A(\alpha[y_1,\dotsc, y_m,1,\dotsc,1])
\widehat{W}_{m,a'}^A(\alpha[y_1, \dotsc, y_m])
\biggl[ \prod_{j=1}^m y_j^{js} \frac{dy_j}{y_j} \biggr].
\end{split}\end{align}
Now, for $ y_j \in\R_+ $ $ (m+1 \le j \le n-1) $, we set
\begin{align*}
&  Z(s,W_{n,a}^A,W_{m,a'}^A; y_{m+1}, \dotsc, y_{n-1})
\\
&  = 2^m \int_{(\R_+)^m}
\widehat{W}_{n,a}^A(\alpha[y_1,\dotsc, y_m,y_{m+1},\dotsc,y_{n-1},1])
\widehat{W}_{m,a'}^A(\alpha[y_1,\dotsc,y_m])
\biggl[ \prod_{j=1}^{m} y_j^{js}
\frac{dy_j}{y_j} \biggr],
\end{align*}so that $$ I_{n,m;a,a'}(s)= Z(s,W_{n,a}^A,W_{m,a'}^A;1,\dotsc,1).$$Applying Mellin inversion to the right hand side yields\begin{align} \label{GLnGLmMellin}
\begin{split}
I_{n,m;a,a'}(s)
& = \frac{1}{(2\pi i)^{n-m-1}} \int_{s_{m+1},\dotsc,s_{n-1}}
 Z(s_{m+1}, \dotsc, s_{n-1}) \,ds_{m+1} \cdots ds_{n-1}
,
\end{split}
\end{align}
where
\begin{align*}
&  Z(s_{m+1},\dotsc, s_{n-1})
\\
&  =
\int_{(\R_+)^{n-m-1}}
 Z(s,W_{n,a}^A,W_{m,a'}^A; y_{m+1}, \dotsc, y_{n-1})
\biggl[ \prod_{j=m+1}^{n-1} y_j^{s_j} \frac{dy_j}{y_j} \biggr]
\\
& = 2^m \int_{(\R_+)^{n-1}}
\widehat{W}_{n,a}^A(\alpha[y_1,\dotsc,y_{n-1},1])
\widehat{W}_{m,a'}^A(\alpha[y_1,\dotsc, y_{m}])
\biggl[ \prod_{j=1}^{m} y_j^{js} \frac{dy_j}{y_j} \biggr]
\biggl[ \prod_{j=m+1}^{n-1} y_j^{s_j} \frac{dy_j}{y_j} \biggr]
\\
& = 2^m \int_{(\R_+)^{n-1}}
\widehat{W}_{n,a}^A(\alpha[y_1,\dotsc,y_{n-1},1])
\widehat{W}_{m,a'}^A(\alpha[y_1,\dotsc, y_{m-1},1])
\biggl[ \prod_{j=1}^{m} y_j^{js} \frac{dy_j}{y_j} \biggr] y_m^{|a'|} 
\biggl[ \prod_{j=m+1}^{n-1} y_j^{s_j} \frac{dy_j}{y_j} \biggr]
\\ &= \frac{2^m}{(2\pi i)^{m-1}} 
   \int_{  s_1,\dotsc,s_{m-1} } 
 U_{n,a}(s-s_1, 2s-s_2, \dotsc, (m-1)s-s_{m-1},
 ms+|a'| ,s_{m+1},\dotsc,s_{n-1})
\\
& \times 
 U_{m,a'}(s_1,\dotsc,s_{m-1}) 
 \,ds_1 \cdots ds_{m-1}  ,
\end{align*}
since, again, the Mellin transform of a product is the convolution of the Mellin transforms.  So we conclude from (\ref{GLnGLmMellin}) that
\begin{align}
\begin{split}
 I_{n,m;a,a'}(s)  &= \frac{2^m}{(2\pi i)^{n-2}} 
   \int_{\scriptstyle s_1,\dotsc,s_{m-1} \atop \scriptstyle s_{m+1},\dotsc,s_{n-1}} 
 U_{n,a}(s-s_1, 2s-s_2, \dotsc, (m-1)s-s_{m-1},
\\
 & \qquad \qquad \qquad \qquad   ms+|a'| ,s_{m+1},\dotsc,s_{n-1})
\\
& \times 
 U_{m,a'}(s_1,\dotsc,s_{m-1}) 
 \,ds_1 \cdots ds_{m-1} ds_{m+1}\cdots ds_{n-1},\label{GLnGLmMellin2}
\end{split}
\end{align}as required.

Finally, we prove part (3).  We have 
 \begin{align*} 
{I}^\vee_{n,n-2;a,a'}(s) & =
Z^{\vee}(s,   (W_{n,a}^A)^{\vee},(\overline{W}_{n-2,a'}^A)^{\vee})
\\
& = 2^{n-2} \int_{(\R_+)^{n-2}} \int_{\R^{n-2}}
(W_{n,a}^A)^{\vee}( M(x_1,\dotsc,x_{n-2}; y_1,\dotsc,y_{n-2},1) )
\\
& \times
(\overline{W}_{n-2,a'}^A)^{\vee} (\alpha[y_1,\dotsc, y_{n-2}]) \biggl[ \prod_{j=1}^{n-2}  dx_j \, y_j^{j(s-n+j+1)} \,\frac{dy_j}{y_j}\biggr],
\end{align*}
where
$$
M(x_1,\dotsc,x_{n-2};y_1,\dotsc, y_{n-1})
= \begin{pmatrix}
\overline{y} & & \\ x & y_{n-1} & \\ & & 1 \end{pmatrix},
$$
with $ x = (x_1, \dotsc, x_{n-2}) $ and $ \overline{y} = y_{n-1} \cdot \alpha[y_1,\dotsc, y_{n-2 }] $.  Then Mellin inversion implies
\begin{align*}
{I}^\vee_{n,n-2;a,a'}(s) = \frac{1}{2\pi i} \int_{w} I^\vee_{n,n-2;a,a'}(s,w) \,dw,
\end{align*}
where
\begin{align}\begin{split}\label{ItildeDef}
 I^\vee_{n,n-2;a,a'}(s,w) &
 = 2^{n-2} \int_{(\R_+)^{n-1}} \int_{\R^{n-2}} 
     (W_{n,a}^A)^{\vee}( M(x_1,\dotsc,x_{n-2}; y_1,\dotsc,y_{n-2},y_{n-1}) ) 
\\
& \times 
     (\overline{W}_{n-2,a'}^A)^{\vee}(\alpha[y_1,\dotsc, y_{n-2}])   y_{n-1}^w\biggl[ \prod_{j=1}^{n-2} dx_j \, y_j^{j(s-n+j+1)}   \, \frac{dy_j}{y_j}\biggr]\frac{dy_{n-1}}{y_{n-1}}.\end{split}
\end{align}To complete the proof of part (3) of the present proposition, then, it will suffice to demonstrate:

\medskip\noindent{\bf Claim.}  {\it 
\begin{align*}
I^\vee_{n,n-2;a,a'}(s,w)=\frac{I_{ n,n-1,-a,a^\vee}\biggl(\ds\frac{3}{2}+  \frac{w+|a'|-1}{n-1}  \biggr)}{ 2\prod_{j=2}^{n-1}
\Gamma_{\R}(a_1^\vee-a_j^\vee+1) }, \end{align*}with $a^\vee=(a_1^{\vee},\dotsc, a_{n-1}^{\vee}) $ as in \hbox{\rm (\ref{acheck}).}}

 \medskip\noindent{\it Proof of Claim.} Noting that\begin{align*}
M(x_1,\dotsc,x_{n-2}; y_1,\dotsc,y_{n-1})
& = \alpha[y_1,\dotsc,y_{n-1},1]\cdot
M \Bigl( \frac{x_1}{y_{n-1}}, \dotsc, \frac{x_{n-2}}{y_{n-1}};1,\dotsc,1 \Bigr),
\end{align*}we substitute $ x_j \to   x_j y_{n-1} $ $ (1 \le j \le n-2) $ into (\ref{ItildeDef}), to get
\begin{align}\begin{split} \label{ItildeDef2}
& I^\vee_{n,n-2;a,a'}(s,w) \\&
 = 2^{n-2} \int_{(\R_+)^{n-1}} \int_{\R^{n-2}} 
     (W_{n,a}^A)^{\vee}\bigl(  \alpha[y_1,\dotsc,y_{n-1},1]\cdot
M ( {x_1} , \dotsc,  {x_{n-2}} ;1,\dotsc,1  )\bigr) 
\\
& \times 
     (\overline{W}_{n-2,a'}^A)^{\vee}(\alpha[y_1,\dotsc, y_{n-2}])   y_{n-1}^{w+n-2}\biggl[ \prod_{j=1}^{n-2} dx_j \, y_j^{j(s-n+j+1)}   \, \frac{dy_j}{y_j}\biggr]\frac{dy_{n-1}}{y_{n-1}}.\end{split}
\end{align}
Next, we check that $ g=M(x_1,\dotsc,x_{n};1,\dotsc,1) $ has Iwasawa decomposition
$g= x(g) y(g) k(g),$ where
\begin{align*}
x(g) = (x(g)_{j,k})_{1 \le j,k \le n}\quad \hbox{with}\quad  x(g)_{j,k} & =
\begin{cases} 1 & 1 \le j=k \le n, \\
-x_j x_k/p_k & 1 \le j<k \le n-2, \\
x_j/p_{n-1} & 1 \le j \le n-2, k= n-1, \\
0 & \mbox{otherwise}; \end{cases}
\end{align*}
\begin{align*}
y(g) &= p\biggl[ \frac{\sqrt{p_1 p_3}}{p_2}, \frac{\sqrt{p_2 p_4}}{p_3}, \dotsc,
\frac{ \sqrt{p_{n-2}p_n}}{p_{n-1}},
\sqrt{p_{n-1}}, 1 \biggr];
\end{align*}
and 
\begin{align*}   k(g) = (k(g)_{j,k})_{1 \le j,k \le n} \quad \hbox{with} \quad
k(g)_{j,k} & = \begin{cases}
\sqrt{p_j} / \sqrt{ p_{j+1} } & 1 \le j = k \le n-2; \\
1 / \sqrt{p_{n-1}} & j=k = n-1; \\
1 & j=k = n; \\
-x_j x_k / \sqrt{p_j p_{j+1}} & 1 \le k < j \le n-2; \\
-x_j / \sqrt{p_j p_{j+1}} & 1 \le j \le n-2, k=n-1; \\
x_k / \sqrt{p_{n-1}} & j=n-1, 1 \le k \le n-2; \\
0 & \mbox{otherwise}. \end{cases}
\end{align*}
Here $ p_j = 1 + \sum_{p=1}^{j-1} x_p^2 $ for $ 1 \le j \le n-1, $ and $ p_n = 1 $.

So (\ref{ItildeDef2}), together with Corollary \ref{GLcont} and the substitutions $x_j\to (-1)^{n-j-2}x_j$ for $1\le j\le n-2$,  give
\begin{align*}
I^\vee_{n,n-2;a,a'}(s,w)
& = 2^{n-2} \int_{(\R_+)^{n-1}} \int_{\R^{n-2}}
\exp \biggl\{ -2\pi i \biggl( \sum_{j=1}^{n-3} \frac{x_j x_{j+1}}{p_{j+1}} y_j
+ \frac{x_{n-2}}{p_{n-1}} y_{n-2} \biggr) \biggr\}
\\
& \times
W_{n,-a}^A \biggl( \alpha \biggl[ \frac{ \sqrt{p_1 p_3}}{p_2} y_1, \dotsc,
\frac{\sqrt{p_{n-2}p_n}}{p_{n-1}} y_{n-2}, \sqrt{p_{n-1}} y_{n-1}, 1 \biggr] \biggr)
\\
& \times W_{n-2,-a'}^A (\alpha[y_1,\dotsc,y_{n-2}])
\\
& \times \biggl[ \prod_{j=1}^{n-2} y_j^{j(s-n+j+1)} \biggr] y_{n-1}^{w+n-2}
\prod_{j=1}^{n-2} dx_j \prod_{j=1}^{n-1} \frac{dy_j}{y_j}.
\end{align*}
We now substitute
$$
y_j \to \frac{p_{j+1}}{\sqrt{p_j p_{j+2}}} y_j \ (1 \le j \le n-2),
\ y_{n-1} \to \frac{1}{\sqrt{p_{n-1}}} y_{n-1},
$$
and collect the powers of the $ y_j $'s and $ q_j $'s, to find that
\begin{align*}
I^\vee_{n,n-2;a,a'}(s,w)
& = 2^{n-2} \int_{(\R_+)^{n-1}} \int_{\R^{n-2}}
\exp \biggl\{ -2\pi i \biggl( \sum_{j=1}^{n-3} \frac{x_j x_{j+1}}{\sqrt{p_jp_{j+2}}} y_j
+ \frac{x_{n-2}}{\sqrt{p_{n-2}}} y_{n-2} \biggr) \biggr\}
\\
& \times \widehat{W}_{n,-a}^A( \alpha[y_1,\dotsc, y_{n-1},1])
\\
& \times \widehat{W}_{n-2,-a'}^A
\biggl(a \biggl[ \frac{p_2}{\sqrt{p_1p_3}} y_1, \dotsc,
\frac{p_{n-2}}{\sqrt{p_{n-3}p_{n-1}}} y_{n-3}, \frac{p_{n-1}}{\sqrt{p_{n-2}}} y_{n-2}
\biggr] \biggr)
\\
& \times \biggl[ \prod_{j=1}^{n-2} y_j^{js} \biggr] y_{n-1}^{w+  {(3n-5)}/{2}}
p_{n-1}^{ {(n-1)(2s-3)}/{4} -  {w}/{2}} \prod_{j=1}^{n-2} \frac{dx_j}{\sqrt{p_j}}
\prod_{j=1}^{n-1} \frac{dy_j}{y_j}.
\end{align*}

Making the change of variables $ (x_1,\dotsc,x_{n-2}) \to (x_1',\dotsc,x_{n-2}') $ given by
$$
x_j' = x_j \sqrt{ \frac{p_{n-1}}{p_j p_{j+1}}},
$$
we find that
$$
x_j = x_j' \sqrt{ \frac{r_1}{r_j r_{j+1}} },
\ \ p_j = \frac{r_1}{r_j}, \ \ \mbox{ and }
\ \ \prod_{j=1}^{n-2} \frac{dx_j}{\sqrt{p_j}} = \prod_{j=1}^{n-2} \frac{dx_j'}{\sqrt{r_{n-j}}},
$$
where
$ r_j = 1 + \sum_{p=j}^{n-2} (x_p')^2 $ for $ 1 \le j \le n-2 $ and $ r_{n-1} = 1 $.
Then we get
\begin{align*}
I^\vee_{n,n-2;a,a'}(s,w)
& = 2^{n-2} \int_{(\R_+)^{n-1}} \int_{\R^{n-2}}
\exp \biggl\{ -2\pi i \biggl( \sum_{j=1}^{n-3} \frac{x_j' x_{j+1}'}{r_{j+1} } y_j
+ x_{n-2}' y_{n-2} \biggr) \biggr\}
\\
& \times \widehat{W}_{n,-a}^A( \alpha[y_1,\dotsc, y_{n-1},1])
\\
& \times \widehat{W}_{n-2,-a'}^A
\biggl(\alpha\biggl[ \frac{\sqrt{r_1 r_3}}{r_2} y_1, \dotsc,
\frac{ \sqrt{r_{n-3}r_{n-1}}}{r_{n-2}} y_{n-3},
\sqrt{ r_{n-2}} y_{n-2} \biggr] \biggr)
\\
& \times
\biggl[ \prod_{j=1}^{n-2} y_j^{js} \biggr]
y_{n-1}^{w+  {(3n-5)}/{2}}
r_1^{ {(n-1)(2s-3)}/{4} -  {(w+|a'|)}/{2}} \prod_{j=2}^{n-2} r_j^{- 1/2}
\prod_{j=1}^{n-2} dx_j'
\prod_{j=1}^{n-1} \frac{dy_j}{y_j}.
\end{align*}
Here, we used the fact that $ \widehat{W}_{n-2,-a'}^A (\alpha[y_1,\dotsc, \sqrt{r_1} y_{n-2}])
= r_1^{-|a'|/2} \widehat{W}_{n-2,-a'}^A (\alpha[y_1,\dotsc, y_{n-2}]) $.
Using Proposition \ref{GLrec2}, we can evaluate the integral  in
the variables $ dx_j' $ as a class one Whittaker function on $GL_{n-1}(\R) $; the result is:
\begin{align*}
I^\vee_{n,n-2;a,a'}(s,w) & =
\frac{ 2^{n-2}}{ \prod_{2 \le j \le n-1} \Gamma_{\R}(a_1^{\vee}-a_j^{\vee}+1) }
\\
& \times
\int_{(\R_+)^{n-1}} \widehat{W}_{n,-a}^A(\alpha[y_1,\dotsc, y_{n-1},1])
\widehat{W}_{n-1,a^{\vee}}^A(\alpha[y_1,\dotsc,y_{n-1}])
\\
& \times \biggl[ \prod_{j=1}^{n-2} y_j^{j(s + {a_1^{\vee}}/{(n-2)})} \biggr]
y_{n-1}^{w+ {(3n-5)}/{2}-|a^{\vee}|} \prod_{j=1}^{n-1} \frac{dy_j}{y_j}.
\end{align*}
Now $ |a^{\vee}| = -|a'| $, so that
\begin{align*}
I^\vee_{n,n-2;a,a'}(s,w)
& = \frac{2^{n-2}}{ \prod_{2 \le j \le n-1} \Gamma_{\R}(a_1^{\vee}-a_j^{\vee}+1)}
\\
& \times \int_{(\R_+)^{n-1}} \widehat{W}_{n,-a}^A(\alpha[y_1,\dotsc, y_{n-1},1])
\widehat{W}_{n-1,a^{\vee}}^A(\alpha[y_1,\dotsc,y_{n-1}])
\\
& \times \prod_{j=1}^{n-1} y_j^{j( {3}/{2}+  {(w+|a'|-1)}/{(n-1)}) }
\prod_{j=1}^{n-1} \frac{dy_j}{y_j}
\\& =\frac{I_{ n,n-1,-a,a^\vee}\biggl(\ds\frac{3}{2}+  \frac{w+|a'|-1}{n-1}  \biggr)}{2 \prod_{j=2}^{n-1}
\Gamma_{\R}(a_1^\vee-a_j^\vee+1) },
\end{align*}
the last step by (\ref{Inem}), and our claim is proved.  This, in turn, completes the proof of our proposition.

\end{proof}
\medskip
\subsection{The cases $ m=n$ and $m=n-1 $}

In  \cite[Theorem 3.4]{St01} and \cite[Theorem 1.1]{St02}, the second author of the present work has proved the following theorem, under the assumption that the characters $\chi_a^A$ and $\chi_{a'}^A$ (cf. section 1.1 above) are trivial on scalar matrices.  Here we present a new proof that  requires no such stipulation,  and that is simplified by the use of new formulas, recently derived by the authors in \cite{ISt}, relating Whittaker functions on $GL(n,\R)$ to those on $GL(n-1,\R)$. (Earlier proofs of these results used different recursive formulas, relating $GL(n,\R)$ Whittaker functions to those on $GL(n-2,\R)$, cf. \cite{St90}).)

\begin{thm}\label{GLGL}
\begin{itemize}
\item[(1)] For $ a = (a_1,\dotsc,a_n) $ and $ a' = (a_1',\dotsc, a_{n-1}') $,
we have
\begin{align*}
I_{n,n-1;a,a'}(s)
= \prod_{1 \le j \le n, 1 \le k \le n-1} \Gamma_{\R}(s+a_j+a_k').
\end{align*}
\item[(2)] For $ a=(a_1,\dotsc,a_n) $ and $ a'= (a_1',\dotsc,a_n') $,
we have
\begin{align*}
I_{n,n;a,a'}(s)
= \prod_{1 \le j \le n, 1 \le k \le n} \Gamma_{\R}(s+a_j+a_k').
\end{align*}
\end{itemize}
\end{thm}

\begin{proof} Our proof here is based on a relation between  
$ I_{n,n;a,a'}(s) $ and $ I_{n,n-1;a,a'}(s) $, and another between  
$ I_{n,n-1;a,a'}(s) $ and $ I_{n,n-1-1;a,a'}(s) $, as follows.

\begin{prop} \label{reGLGL}
\begin{itemize}
\item[(1)] For $ a = (a_1,\dotsc,a_n) $ and $ a' = (a_1',\dotsc,a_n') $,
we have
\begin{align*}
I_{n,n;a,a'}(s)
= \biggl[ \prod_{j=1}^n \Gamma_{\R}(s+a_1'+a_j) \biggr]
I_{n,n-1;a,\widetilde{a}' }\Bigl(s-\frac{a_1'}{n-1} \Bigr),
\end{align*}
where $ \widetilde{a}' = (\widetilde{a}_1',\dotsc,\widetilde{a}_{n-1}') $ with
$ \widetilde{a}'_j = a'_{j+1} + a_1'/(n-1) $.
\item[(2)]
For $ a = (a_1, \dotsc,a_n) $ and $ a' = (a_1',\dotsc,a_{n-1}') $, we have
\begin{align*}
I_{n,n-1;a,a'}(s)
= \biggl[ \prod_{j=1}^{n-1} \Gamma_{\R}(s+a_1+a_j') \biggr]
I_{n,n-1-1;\widetilde{a},a'} \Bigl(s-\frac{a_1}{n-1} \Bigr),
\end{align*}
where $ \widetilde{a} = (\widetilde{a}_1,\dotsc,\widetilde{a}_{n-1}) $ with
$ \widetilde{a}_j = a_{j+1} + a_1/(n-1) $.
\end{itemize}
\end{prop}


\medskip\noindent{\it Proof of Proposition.}  Let us show (2); the proof of (1) is similar.

By Propositions \ref{Barnesforms} (2) (in the case $m=n-1$) and \ref{MBGL}, we have
\begin{align*}
I_{n,n-1;a,a'}(s)
& = \frac{ 1 }{(2\pi i)^{2n-4}}
\int_{\scriptstyle s_1,\dotsc,s_{n-2} \atop \scriptstyle z_1,\dotsc,z_{n-2}}
U_{n-1,a'}(s_1,\dotsc,s_{n-2})
\, U_{n-1,\widetilde{a}}(z_1,\dotsc,z_{n-2})
\\
& \times
\biggl[ \prod_{j=1}^{n-1}
\Gamma_{\R} \Bigl( js-s_j-z_{j-1}+ \frac{(n-j)a_1}{n-1} \Bigr)
\Gamma_{\R} \Bigl( js-s_j-z_j - \frac{ja_1}{n-1} \Bigr) \biggr]
\\
& \times dz_1 \cdots dz_{n-2} \,ds_1 \cdots ds_{n-2}.
\end{align*}
Here we understand that $ s_{n-1} = -|a'| $ and
$ z_{n-1} = -|a| $.
We substitute $ z_j \to z_j - ja_1/(n-1) $ for $ 1 \le j \le n-2 $, to find
\begin{align*}
I_{n,n-1;a,a'}(s)
& = \frac{ 1 }{(2\pi i)^{2n-4}}
\int_{\scriptstyle s_1,\dotsc,s_{n-2} \atop \scriptstyle z_1,\dotsc,z_{n-2}}
U_{n-1,a'}(s_1,\dotsc,s_{n-2})
\\
& \times
U_{n-1,\widetilde{a}} \Bigl( z_1-\frac{a_1}{n-1},\dotsc,z_{n-2}-\frac{(n-2)a_1}{n-1} \Bigr)
\\
& \times
\biggl[ \prod_{j=1}^{n-2}
\Gamma_{\R} ( js-s_j-z_{j-1}+ a_1 )
\Gamma_{\R} ( js-s_j-z_j ) \biggr]
\\
& \times \Gamma_{\R} ( (n-1)s-z_{n-2}+a_1+|a'|) \Gamma_{\R}( (n-1)s-a_1+|a|+|a'|)
\\
& \times dz_1 \cdots dz_{n-2} \,ds_1 \cdots ds_{n-2}.
\end{align*}
We now apply Proposition \ref{lemGL} (1), with
$ p_j = js-z_j $ for $ 1 \le j \le n-2 $  and  $\sigma = s+a_1 $, to integrate with
respect to the $ s_j $'s and get:
\begin{align*}
I_{n,n-1;a,a'}(s)
& = \Gamma_{\R}( (n-1)s-a_1+|a|+|a'|)
\prod_{j=1}^{n-1} \Gamma_{\R}(s+a_1+a_j')
\\
& \times
\frac{ 2^{n-2} }{(2\pi i)^{n-2}} \int_{ \scriptstyle z_1,\dotsc,z_{n-2}}
U_{n-1,a'}(s-z_1,\dotsc,(n-2)s-z_{n-2})
\\
& \times
U_{n-1,\widetilde{a}} \Bigl( z_1-\frac{a_1}{n-1},\dotsc,z_{n-2}-\frac{(n-2)a_1}{n-1} \Bigr)
\,dz_1 \cdots dz_{n-2}.
\end{align*}
If we substitute $ z_j \to z_j+ja_1/(n-1) $ for $1\le j\le n-2$, we find that
\begin{align*}
& I_{n,n-1;a,\widetilde{a}}(s) = \biggl[ \prod_{j=1}^{n-1} \Gamma_{\R}(s+a_1+a_j') \biggr]
\Gamma_{\R} \Bigl( (n-1)\Bigl(s-\frac{a_1}{n-1}\Bigr) + |a|+|a'| \Bigr)
\\
& \times
\frac{2^{n-2}}{(2\pi i)^{n-2}} \int_{z_1,\dotsc,z_{n-2}}
U_{n-1,a} \Bigr( \Bigl(s-\frac{a_1}{n-1} \Bigr)-z_1,
\dotsc, (n-2)\Bigl(s-\frac{a_1}{n-1}\Bigr) -z_{n-2} \Bigr)
\\
& \times
U_{n-1,\widetilde{a}} ( z_1, \dotsc, z_{n-2})
\,dz_1 \cdots dz_{n-2},
\end{align*}
and our proposition is proved.

\medskip

To deduce Theorem \ref{GLGL}, we use induction.  First, we consider part (1) of the theorem:  in the case $n=2$, this amounts to Barnes' first lemma.

Now by Proposition \ref{reGLGL} (2), we have
\begin{align*}
I_{n,n;a,a'}(s)
& = \prod_{j=1}^n \Gamma_{\R}(s+a_1'+a_j)
\prod_{j=1}^{n-1} \Gamma_{\R}
\Bigl( s-\frac{a_1'}{n-1} + a_1 + \widetilde{a}_j' \Bigr)
\\
& \times
I_{n-1,n-1;\widetilde{a}, \widetilde{a}'} \Bigl(s- \frac{a_1+a_1'}{n-1} \Bigr).
\end{align*}
By the  induction hypothesis applied to  $ I_{n,n;a,a'}(s) $, we find that
\begin{align*}
I_{n,n;a,a'}(s)
& = \prod_{j=1}^n \Gamma_{\R}(s+a_1'+a_j)
\prod_{j=1}^{n-1} \Gamma_{\R}(s+a_1+a_{j+1}')
\\
& \times
\prod_{ 1 \le j \le n-1, 1 \le k \le n-1}
\Gamma_{\R} \Bigl( s-\frac{a_1+a_1'}{n-1}+ \widetilde{a}_j + \widetilde{a}_k' \Bigr)
\\
& = \prod_{j=1}^n \Gamma_{\R}(s+a_1'+a_j)
\prod_{j=1}^{n-1} \Gamma_{\R}(s+a_1+a_{j+1}')
\\
& \times
\prod_{1 \le j \le n-1, 1 \le k \le n-1}
\Gamma_{\R}(s+a_{j+1}+a_{k+1}')
\end{align*}
which is the desired result.

We now turn to part (2) of our theorem.  The case $n=2$ follows from Proposition \ref{Barnesforms} (2) (in the case $n=2,m=1$).
The general case then follows by  Proposition \ref{reGLGL} (2) and by induction (much as in the above proof for the case $m=n$),
and thus we complete our proof of Theorem \ref{GLGL}.
 \end{proof}

\subsection{The case  $ m=n-2 $}

We investigate the zeta integrals $ Z(s, W, W') $ and $ Z^{\vee}(s, W, W') $, as defined in section 2.1 above, in the case $ m=n-2 $.
Note that the case $ (n,m) = (3,1) $ has previously been treated in \cite{HM}.

\begin{thm} \label{thmnminustwo} For $ a =(a_1,\dotsc,a_n) $ and $ a' = (a_1',\dotsc,a_{n-2}') $,
we have
\begin{align*}
I_{n,n-2;a,a'}(s) =
\biggl[
\prod_{1 \le j \le n, 1 \le k \le n-2} \Gamma_{\R}(s+a_j+a_k')
\biggr]
\cdot
\frac{1}{4\pi i} \int_w
\frac{ \prod_{j=1}^n \Gamma_{\R}(w-a_j) }{\prod_{j=1}^{n-2} \Gamma_{\R}(s+w+a_j')}
\,dw.
\end{align*}
\end{thm}

\bpf
We evaluate the integral (\ref{GLnGLmMellin2}) with $ m =n-2 $.
Using Proposition \ref{MBGL} to rewrite $ U_{n,a} $, we have
\begin{align*}
I_{n,n-2;a,a'}(s)
& = \frac{2^{-1}}{(2 \pi i)^{2n-4}}
\int_{\scriptstyle s_1,\cdots,s_{n-3},s_{n-1} \atop z_1,\dotsc,z_{n-2}}
U_{n-1,\widetilde{a}} (z_1,\dotsc, z_{n-2} )
U_{n-2,a'}(s_1,\dotsc, s_{n-3})
\\
& \times
\biggl[ \prod_{j=1}^{n-3}
\Gamma_{\R} \Bigl( j s-s_j-z_j-\frac{ja_1}{n-1} \Bigr)
\Gamma_{\R} \Bigl( j s-s_j-z_{j-1} + \frac{(n-j)a_1}{n-1} \Bigr) \biggr]
\\
& \times
\Gamma_{\R} \Bigl( (n-2)s+|a'|-z_{n-2} -\frac{(n-2)a_1}{n-1} \Bigr)
\\
& \times
\Gamma_{\R} \Bigl( (n-2)s+|a'|-z_{n-3} + \frac{2a_1}{n-1} \Bigr)
\\
& \times
\Gamma_{\R}(s_{n-1} + |a|-a_1) \Gamma_{\R} \Bigl( s_{n-1}-z_{n-2}+ \frac{a_1}{n-1} \Bigr)
\\
& \times
dz_1 \cdots dz_{n-2} \, ds_1 \cdots ds_{n-3} ds_{n-1}.
\end{align*}
We apply Proposition \ref{lemGL} (1) with $ \sigma = s+a_1 $
to integrate with respect to $ s_1, \dotsc, s_{n-3} $:
\begin{align*}
I_{n,n-2;a,a'}(s)
& = \prod_{j=1}^{n-2} \Gamma_{\R}(s+a_1+a_j')
\cdot \frac{ 2^{n-4} }{(2\pi i)^{n-1}} \int_{s_{n-1},z_1,\cdots,z_{n-2}}
U_{n-1,\widetilde{a}}(z_1,\dotsc, z_{n-2})
\\
& \times
U_{n-2,a'} ( s'-z_1, \dotsc, (n-3) s'-z_{n-3} ) \cdot
\Gamma_{\R} ( (n-2)s'+|a'|-z_{n-2} )
\\
& \times
\Gamma_{\R}(s_{n-1} + |a|-a_1) \Gamma_{\R} \Bigl( s_{n-1}-z_{n-2}+ \frac{a_1}{n-1} \Bigr)
\, dz_1 \cdots dz_{n-2} \,ds_{n-1}.
\end{align*}
Here we have defined $ s' = s-a_1/(n-1). $

Next we use Proposition \ref{lemGL} (1), with $ \sigma = s+s_{n-1}+|a| $, to rewrite $ U_{n-2,a'} $;
we find that
\begin{align*}
I_{n,n-2;a,a'}(s)
& = \prod_{j=1}^{n-2} \Gamma_{\R}(s+a_1+a_j') \cdot
\frac{2^{-1} }{(2\pi i)^{2n-4}} \int_{ \scriptstyle s_{n-1}, z_1, \dotsc,z_{n-2} \atop
\scriptstyle w_1, \dotsc, w_{n-3} }
U_{n-1,\widetilde{a}}(z_1,\dotsc,z_{n-2})
\\
& \times \frac{ \Gamma_{\R}((n-3)s'-z_{n-3}+s+s_{n-1}+|a|+|a'|) }
{ \prod_{j=1}^{n-2} \Gamma_{\R}(s+s_{n-1}+|a|+a_j') }
\\
& \times
U_{n-2,a'}(w_1,\dotsc, w_{n-3})
\cdot \prod_{j=1}^{n-3} \Gamma_{\R}( j s'-z_j-w_j )
\\
& \times
\prod_{j=1}^{n-3}
\Gamma_{\R} ( (j-1) s' -z_{j-1}-w_j+s+s_{n-1}+|a| )
\\
& \times
\Gamma_{\R} ( (n-2)s'+|a'|-z_{n-2} )
\Gamma_{\R}(s_{n-1} + |a|-a_1)
\\
& \times
\Gamma_{\R} \Bigl( s_{n-1}-z_{n-2}+ \frac{a_1}{n-1} \Bigr)
\, dw_1\cdots dw_{n-3} \,dz_1 \cdots dz_{n-2} \,ds_{n-1}.
\end{align*}
If we define $ w_{n-2} = - |a'| $ and $ w_{n-1} = (n-1)s' +|a| $,
the above can be rewritten as follows:
\begin{align*}
I_{n,n-2:a,a'}(s)
& = \prod_{j=1}^{n-2} \Gamma_{\R}(s+a_1+a_j')
\cdot \frac{2^{-1} }{(2\pi i)^{2n-4}} \int_{ \scriptstyle s_{n-1}, z_1, \dotsc,z_{n-2} \atop
\scriptstyle w_1, \dotsc, w_{n-3} }
U_{n-1,\widetilde{a}}(z_1,\dotsc,z_{n-2})
\\
& \times U_{n-2,a'}(w_1,\dotsc, w_{n-3})
\cdot \frac{ \Gamma_{\R}(s_{n-1} + |a|-a_1) \Gamma_{\R}( s+s_{n-1}+|a|-w_1) }
{ \prod_{j=1}^{n-2} \Gamma_{\R}(s+s_{n-1}+|a|+a_j') }
\\
& \times
\prod_{j=1}^{n-2} \Gamma_{\R} ( j s' -z_j-w_j )
\Gamma_{\R} \Bigl( (j+1) s' -z_j
-w_{j+1} + s_{n-1} + \frac{a_1}{n-1} + |a| \Bigr)
\\
& \times
dw_1\cdots dw_{n-3} \,dz_1 \cdots dz_{n-2} \,ds_{n-1}.
\end{align*}
We can apply Proposition \ref{lemGL} (2) with
$ p_j = js' -w_j $ $ (1 \le j \le n-2) $ and
$ \sigma = s_{n-1}+a_1/(n-1)+|a| $, to integrate with respect to $ z_1,\dotsc, z_{n-2} $
(note that the condition $p_{n-1} = -|\widetilde{a}| = -|a| $ is satisfied because of the
definition of $ w_{n-1} $), and arrive at
\begin{align*}
I_{n,n-2;a,a'}(s)
& = \prod_{j=1}^{n-2} \Gamma_{\R}(s+a_1+a_j')
\cdot \frac{2^{n-3} }{(2\pi i)^{n-2}} \int_{ s_{n-1}, w_1, \dotsc, w_{n-3} }
U_{n-2,a'}(w_1,\dotsc, w_{n-3})
\\
& \times \frac{ \Gamma_{\R}(s_{n-1} + |a|-a_1)
\prod_{j=1}^{n-1} \Gamma_{\R} (s_{n-1}+ \frac{a_1}{n-1}+|a|-\widetilde{a}_j) }
{ \prod_{j=1}^{n-2} \Gamma_{\R}(s+s_{n-1}+|a|+a_j') }
\\
& \times
U_{n-1,\widetilde{a}}( s'-w_1, \dotsc,
(n-3) s' -w_{n-3},
(n-2) s' + |a'| )
\\
& \times
dw_1\cdots dw_{n-3} \,ds_{n-1}\\
& = \biggl[ \prod_{j=1}^{n-2} \Gamma_{\R}(s+a_1+a_j') \biggr]
\cdot I_{n-1,n-2; \widetilde{a},a'} (s')
\\
& \times
\frac{1}{4 \pi i} \int_{s_{n-1}}
\frac{ \prod_{j=1}^n \Gamma_{\R}(s_{n-1}+|a|-a_j) }
{ \prod_{j=1}^{n-2} \Gamma_{\R}(s+s_{n-1}+|a|+a_j')} \,ds_{n-1},
\end{align*}the last step by (\ref{GLnGLmMellin2}).
Hence Theorem \ref{GLGL} (1) yields our assertion.
\epf

\medskip


Now as may be seen, for example, in \cite{JS3},  the $GL_n\times GL_{n-2}$ functional equation at a real place $\nu$ relates $Z(s,W,W')$ to $Z^{\vee}(s, R(w_{n,n-2})  W ^{\vee},(W'  )^{\vee}).
$
Here $ w_{n,m} = \begin{pmatrix} I_m & \\ & w_{n-m} \end{pmatrix} $ and
$ R $ means right translation.  In the situation of interest to us here, where $W=W^A_{n,a}$ and $W'=\overline{W}^A_{n-2,a'}$ as above,   we may ignore the $ R(w_{n,n-2}) $ in $Z^{\vee}(s, R(w_{n,n-2}) W ^{\vee},(W'  )^{\vee})$,  since $ W_{n,a}^A $ is right $ O(n) $-invariant.
The relevant functional equation then relates $   Z(s, W_{n,a}^A, \overline{W}_{n-2,a'}^A)=I_{n,n-2;a,a'}(s)  $ to 
$ 
Z^{\vee}(s,  (W_{n,a}^A)^{\vee},(\overline{W}_{n-2,a'}^A)^{\vee})=I^\vee_{n,n-2;a,a'}(s)  .
$  

More specifically:  the functional equation demonstrated in \cite{JS3} amounts, in the present context, to the assertion that$$\frac{I_{n,n-2;a,a'}(s)}{\prod_{1 \le j \le n, 1 \le k \le n-2} \Gamma_{\R}(s+a_j+a_k')}  =\frac{I^\vee_{n,n-2;a,a'}(1-s)}{\prod_{1 \le j \le n, 1 \le k \le n-2} \Gamma_{\R}(1-s-a_j-a_k')}. $$
We now wish to show that the above functional equation may also be derived using results contained herein.  To do so it will suffice, in light of Theorem \ref{thmnminustwo}, to prove the following:
\begin{thm}  \label{GLunip}
We have
\begin{align*}
 {I}^\vee_{n,n-2;a,a'}(s)
& = \biggl[ \prod_{1 \le j \le n, 1 \le k \le n-2} \Gamma_{\R}(s-a_j-a_k') \biggr]
\\
& \times \frac{1}{4 \pi i} \int_{w}
\frac{ \prod_{1 \le j \le n} \Gamma_{\R} (w-a_j) }
{ \prod_{1 \le j \le n-2} \Gamma_{\R}(1-s+w+a_j')} \,dw.
\end{align*}
\end{thm}
\begin{proof}  By Proposition \ref{Barnesforms} (3),
\begin{align*} 
{I}^\vee_{n,n-2;a,a'}(s)  =  \frac{1}{4\pi i} \int_{w} \frac{ I_{ n ,n-1,-a,a^\vee}\biggl(\ds\frac{3}{2}+  \frac{w+|a'|-1}{n-1}  \biggr)}{ \prod_{j=2}^{n-1}
\Gamma_{\R}(a_1^\vee-a_j^\vee+1) } \,dw,\end{align*} where
$ a^{\vee} = (a_1^{\vee},\dotsc, a_{n-1}^{\vee}) $ is as in  (\ref{acheck}).  We may now apply Theorem \ref{GLGL} (1) to the archimedean zeta factor inside the integral; we find that
\begin{align*} 
{I}^\vee_{n,n-2;a,a'}(s)  &=  \frac{1}{4\pi i} \int_{w} \frac{ \prod_{1 \le j \le n, 1 \le k \le n-1} \Gamma_{\R}\ds\biggl(\frac{3}{2}+  \frac{w+|a'|-1}{n-1}  -a_j+a_k^\vee\biggr)}{ \prod_{j=2}^{n-1}
\Gamma_{\R}(a_1^\vee-a_j^\vee+1) } \,dw
\\&=  \biggl[ \prod_{1 \le j \le n, 2 \le k \le n-1} \Gamma_{\R}\ds\biggl(\frac{3}{2}+  \frac{w+|a'|-1}{n-1}  -a_j+a_k^\vee\biggr) \biggr]\\&\times\frac{1}{4\pi i} \int_{w} \frac{ \prod_{1 \le j \le n } \Gamma_{\R}\ds\biggl(\frac{3}{2}+  \frac{w+|a'|-1}{n-1}  -a_j+a_1^\vee\biggr)}{ \prod_{j=2}^{n-1}
\Gamma_{\R}(a_1^\vee-a_j^\vee+1) } \,dw
\\&=  \biggl[ \prod_{1 \le j \le n, 2 \le k \le n-1} \Gamma_{\R}\ds (s  -a_j-a_{k-1}'  ) \biggr]\\&\times\frac{1}{4\pi i} \int_{w} \frac{ \prod_{1 \le j \le n } \Gamma_{\R}\ds\biggl(w-a_j-(n-2)s+|a'|+\frac{3n-5}{2} \biggr)}{ \prod_{j=2}^{n-1}
\Gamma_{\R}\biggl(1 +w+a_{j-1}'-(n-1)s+|a'|+\ds\frac{3n-5}{2} \biggr) } \,dw.\end{align*}
We substitute $ {w} \to w +(n-2)s-|a'|-\ds\frac{3n-5}{2} $; we also rearrange the products a bit (letting $k\to k+1$ in the product outside the integral; and letting $j\to j+1$ in the product  in the denominator of the integrand). The result is exactly the statement of Theorem \ref{GLunip}, and we are done.
\end{proof}

\section{archimedean zeta integrals on $SO_{2n+1} \times GL_m $}

Global Rankin-Selberg integrals for $L$-functions on $ SO_{2n+1} \times GL_m $
have been   constructed in \cite{GPR} for the cases  $ m = n$ and $m=n+1$,  in \cite{Gi} for the cases  $m\le n,$ and in \cite{So1} for the cases $m\ge n+1$. In this section, we will consider the cases $m=n-1$, $m=n$, and $m=n+1$.

As may be seen from the works cited above, the archimedean zeta integrals $ J_{\ell,n;a,b}(s) $ of interest to us, for the cases of $SO_{2n+1}\times GL_{n+1}$, $SO_{2n+1}\times GL_{n}$, and $SO_{2n+1}\times GL_{n-1}$ respectively, may be given by:
\begin{align*}
 J _{1,n;a,b}(s) = 2^n \int_{(\R_+)^n} 
   W_{n+1,a}^A(\alpha[y_1,\dotsc, y_n,1]) W_{n,b}^B(\beta[y_1,\dotsc,y_n])
   \biggl[ \prod_{j=1}^{n} y_j^{j(s- ({3n+1})/{2}+j)} \frac{dy_j}{y_j} \biggr],
\end{align*}
\begin{align*}
 J_{0,n;a,b}(s) = 2^n \int_{(\R_+)^n}  
   W_{n,a}^A(\alpha[y_1,\dotsc, y_n]) W_{n,b}^B(\beta[y_1,\dotsc,y_n]) 
   \biggl[ \prod_{j=1}^{n} y_j^{j(s- {3n}/{2}+j)} \frac{dy_j}{y_j} \biggr],
\end{align*}
and 
\begin{align*}
 J_{-1,n;a,b}(s) &= 2^n \int_{\R^{n-1}} \int_{(\R_+)^{n-1}} 
  W_{n-1,a}^A(\alpha[y_1,\dotsc, y_{n-1}])  \\
 & \times 
  W_{n,b}^B( u[x_1,\dotsc,x_{n-1}]  \beta[y_1,\dotsc,y_{n-1},1] ) 
 \biggl[ \prod_{j=1}^{n-1} y_j^{j(s-(3n-3)/{2}+j)} \frac{dy_j}{y_j} \biggr]
   \prod_{j=1}^{n-1} dx_j,
\end{align*}
where
\begin{align*}
 u[x_1,\dotsc,x_{n-1}] = \begin{pmatrix} u & & \\ & 1 & \\ & & {}^t u^{-1} \end{pmatrix},
 \mbox{ with } 
 u = \begin{pmatrix} 
  1 & & & \\ & {\ddots} & & \\ & & 1 & \\ x_1 & \cdots & x_{n-1} & 1 
  \end{pmatrix}.
\end{align*}

In this section, we show the coincidence of the archimedean zeta integrals $ J_{1,n;a,b}(s)$ and $ J_{0,n;a,b}(s)$ with
their associated $L$-factors. We further demonstrate that the archimedean zeta integral $ J_{ -1,n;a,b}(s)$ equals
the associated $L$-factor time a certain integral of Barnes type.

To obtain the desired resultsfor $ J_{1,n;a,b}(s)$ and $ J_{0,n;a,b}(s)$, we will need the following proposition.

\begin{prop} \label{SOGLlem}
\begin{itemize} 
\item[(1)] For $ a = (a_1,\dotsc,a_n) $ and $ b = (b_1,\dotsc, b_n) $, we have
\begin{align*}
 J_{0,n;a,b}(s) = \biggl[ \prod_{j=1}^n \Gamma_{\R}(s+a_j+b_n) \Gamma_{\R}(s+a_j-b_n) \biggr]
   J_{1,n-1;a,\widetilde{b}}(s),
\end{align*}
where $ \widetilde{b} = (b_1,\dotsc,b_{n-1}) $.
\item[(2)] For $ a=(a_1,\dotsc,a_{n+1}) $ and $ b = (b_1,\dotsc, b_n) $, we have 
\begin{align*}
 J_{1,n;a,b}(s) = \frac{ \prod_{j=1}^n \Gamma_{\R}(s+a_1+b_j) \Gamma_{\R}(s+a_1-b_j) }
    { \prod_{j=1}^n \Gamma_{\R}(2s+a_1+a_{j+1}) } \cdot 
   J_{0,n;\widetilde{a},b} \Bigl( s- \frac{a_1}{n} \Bigr),
\end{align*}
where $ \widetilde{a} = (\widetilde{a}_1,\dotsc, \widetilde{a}_{n}) $ with 
$ \widetilde{a}_j= a_{j+1} + ja_1/n $.
\end{itemize} 
\end{prop} 

\begin{proof}  We first prove part (1). By Mellin inversion, we  find that 
\begin{align*}
 J_{0,n;a,b}(s) 
& = \frac{2^n }{(2\pi i)^{n-1}} \int_{s_1,\dotsc,s_{n-1}}
    U_{n,a}(s_1,\dotsc, s_{n-1}) 
 V_{n,b}(s-s_1,\dotsc,(n-1)s-s_{n-1},ns+|a|) \\
& \qquad \qquad \qquad \times 
   ds_1 \cdots ds_{n-1}.
\end{align*}

Then by Proposition \ref{MBSO}, we have
\begin{align*}
  J_{0,n;a,b}(s) 
& = \frac{2^{2n-2} }{(4\pi i)^{3n-3}} \int_{s_1,\dotsc,s_{n-1}} 
   \int_{\scriptstyle w_1,\dotsc, w_{n-1} \atop \scriptstyle z_1,\dotsc,z_{n-1} }
   U_{n,a}(s_1,\dotsc, s_{n-1})  
   V_{n-1,\widetilde{b}} (z_1,\dotsc,z_{n-1})   
\\
& \times 
  \biggl[ \prod_{j=1}^{n-1} \Gamma_{\R}(js-s_j-w_j) \Gamma_{\R}(js-s_j-w_{j-1}-b_n) \biggr]
\\
& \times 
  \biggl[ \prod_{j=1}^{n-1} 
   \Gamma_{\R}(w_j-z_j) \Gamma_{\R}(w_j-z_{j-1}+b_n) \biggr]
\\
& \times  
  \Gamma_{\R}(ns-w_{n-1}+|a|-b_n) 
  \Gamma_{\R}(ns-z_{n-1}+|a|+b_n)
\\
& \times 
  dz_1 \cdots dz_{n-1} \, dw_1 \cdots dw_{n-1} \,ds_1 \cdots ds_{n-1}.
\end{align*}
We apply Proposition \ref{lemGL} (1) with 
$ p_j = js-w_j $, for $1\le j\le n-1$, and $ \sigma = s-b_n $, to perform the integration with respect to 
$ s_1,\dotsc, s_{n-1} $, whence:
\begin{align*}
 J_{0,n;a,b}(s)
& = \biggl[ \prod_{j=1}^n \Gamma_{\R}(s+a_j-b_n) \biggr] 
\\
& \times
   \frac{2^{2n-2} }{(4\pi i)^{2n-2}} 
   \int_{\scriptstyle w_1,\dotsc, w_{n-1} \atop \scriptstyle z_1,\dotsc,z_{n-1} }
   U_{n,a}(s-w_1,\dotsc, (n-1)s-w_{n-1})  
   V_{n-1,\widetilde{b}}(z_1,\dotsc,z_{n-1}) 
\\
& \times 
  \biggl[ \prod_{j=1}^{n-1} \Gamma_{\R}(w_j-z_j) \Gamma_{\R}(w_j-z_{j-1}+b_n) \biggr]
\\
& \times 
  \Gamma_{\R}(ns-z_{n-1}+|a|+b_n) 
  \, dz_1 \cdots dz_{n-1} \, dw_1 \cdots dw_{n-1}.
\end{align*}
To integrate with respect to $ w_1,\dotsc, w_{n-1} $, we substitute 
$ w_j \to js-w_j $ ($1\le j\le n-1$), and apply Proposition \ref{lemGL} (1) again -- this time, with 
$ p_j = js-z_j $ ($1\le j\le n-1$) and $ \sigma = s+b_n $. We get:
\begin{align*}
 J_{0,n;a,b}(s) 
& = \biggl[ \prod_{j=1}^n \Gamma_{\R}(s+a_j+b_n) \Gamma_{\R}(s+a_j-b_n) \biggr] 
\\
& \times 
  \frac{2^{2n-2} }{(4\pi i)^{n-1}} \int_{z_1,\dotsc,z_{n-1}}
   U_{n,a}(s-z_1,\dotsc, (n-1)s-z_{n-1}) 
\\
& \times 
   V_{n-1,\widetilde{b}}(z_1,\dotsc,z_{n-1}) \,dz_1 \cdots dz_{n-1}
\\
& = \biggl[ \prod_{j=1}^n \Gamma_{\R}(s+a_j+b_n) \Gamma_{\R}(s+a_j-b_n) \biggr] 
    J_{1,n-1 ;a,\widetilde{b}}(s),
\end{align*}the last step by the definition of $J_{ 1,n-1;a,\widetilde{b}}(s)$ and by Mellin inversion.
Thus we have finished the proof of (1).

\medskip 

\noindent 
We now prove part (2). By Proposition \ref{MBGL}, we have 
\begin{align*}
 J_{1,n;a,b}(s) 
& = \frac{2^{2n-1}}{(4\pi i)^{2n-1}} \int_{\scriptstyle s_1,\dotsc, s_{n} \atop 
  \scriptstyle z_1,\dotsc,z_{n-1}} 
   U_{n,\widetilde{a}}(z_1,\dotsc,z_{n-1}) 
   V_{n,b}(s_1,\dotsc,s_n) 
\\
& \times 
  \biggl[ \prod_{j=1}^n \Gamma_{\R} \Bigl(js-s_j-z_j- \frac{ja_1}{n} \Bigr)
     \Gamma_{\R} \Bigl( js-s_j-z_{j-1}+ \frac{(n+1-j)a_1}{n} \Bigr) \biggr]
\\
& \times 
  dz_1 \cdots dz_{n-1} \,ds_1 \cdots ds_n.
\end{align*}
Here $ z_n = -|a| = -(a_1+\cdots+a_{n+1}) $.
We apply Proposition \ref{lemGL} (1), with 
$ p_j = z_j $ $ (1 \le j \le n-1) $ and $ \sigma = 2s+(n-1)a_1/n, $ to get 
\begin{align*}
 J_{1,n;a,b}(s) 
& =  \frac{1}{ \prod_{j=1}^n \Gamma_{\R}(2s+a_1+a_{j+1}) } 
\\
& \times 
    \frac{ 2^{2n-1} }{(4\pi i)^{3n-2}} 
   \int_{\scriptstyle s_1,\dotsc, s_{n} \atop \scriptstyle z_1,\dotsc,z_{n-1}, 
     \scriptstyle q_1,\dotsc, q_{n-1}}  
    U_{n,\widetilde{a}}(q_1,\dotsc,q_{n-1}) V_{n,b}(s_1,\dotsc,s_n) 
\\
& \times \Gamma_{\R} \Bigl( 2s+ \frac{(n-1)a_1}{n} + z_{n-1}+ |a|  \Bigr)
\\
& \times 
  \biggl[ \prod_{j=1}^{n-1} \Gamma_{\R}(z_j-q_j) 
        \Gamma_{\R}\Bigl(z_{j-1}-q_j+2s+\frac{(n-1)a_1}{n} \Bigr) \biggr]
\\
& \times 
  \biggl[ \prod_{j=1}^n \Gamma_{\R} \Bigl(js-s_j-z_j- \frac{ja_1}{n} \Bigr)
     \Gamma_{\R} \Bigl( js-s_j-z_{j-1}+ \frac{(n+1-j)a_1}{n} \Bigr) \biggr]
\\
& \times 
   dq_1 \cdots dq_{n-1}\, dz_1 \cdots dz_{n-1}\, ds_1 \cdots ds_n.
\end{align*}
We substitute $ z_j \to -z_j +j(s-a_1/n) $ 
for $ 1 \le j \le n-1 $, to find
\begin{align*}
J_{1,n;a,b}(s) 
& =  \frac{1}{ \prod_{j=1}^n \Gamma_{\R}(2s+a_1+a_{j+1}) } 
\\
& \times \frac{2^{2n-1} }{(4\pi i)^{3n-2}}
  \int_{\scriptstyle s_1,\dotsc, s_{n} \atop \scriptstyle z_1,\dotsc,z_{n-1}, 
     \scriptstyle q_1,\dotsc, q_{n-1}}  
    U_{n,\widetilde{a}}(q_1,\dotsc, q_{n-1})
    V_{n,b}(s_1,\dotsc,s_n) 
\\
& \times \Gamma_{\R}((n+1)s-z_{n-1}+ |a|)
\\
& \times 
  \biggl[ \prod_{j=1}^{n-1} 
  \Gamma_{\R} \Bigl(j \Bigl(s-\frac{a_1}{n}\Bigr) -q_j-z_j \Bigr) 
  \Gamma_{\R} \Bigl(j \Bigl(s-\frac{a_1}{n} \Bigr)-q_j-z_{j-1}+s+a_1 \Bigr) 
    \biggr]
\\
& \times 
  \biggl[ \prod_{j=1}^{n-1} \Gamma_{\R}(z_j-s_j) \biggr]
  \Gamma_{\R}(ns-s_n-a_1+ |a|) 
  \bigg[ \prod_{j=1}^{n} \Gamma_{\R}(z_{j-1}-s_j+s+a_1) \biggr]
\\
& \times 
   dq_1 \cdots dq_{n-1} \,dz_1 \cdots dz_{n-1} \,ds_1 \cdots ds_n.
\end{align*}
Finally, we apply Proposition \ref{lemSO} with 
$ p_j = j(s-a_1/n)-q_j $ $ (1 \le j \le n-1) $,
$ p_n = ns-a_1+|a| $, and $ \sigma = s+a_1 $,
to integrate with respect to the $ s_j $'s and the $ z_j $'s, whence: 
\begin{align*}
  J_{1,n;a,b}(s) 
& =  \frac{\prod_{j=1}^n \Gamma_{\R}(s+a_1+b_j) \Gamma_{\R}(s+a_1-b_j)}
     { \prod_{j=1}^n \Gamma_{\R}(2s+a_1+a_{j+1}) } 
\\
& \times \frac{2^{2n-1} }{(4\pi i)^{n-1}}
  \int_{ \scriptstyle q_1,\dotsc, q_{n-1}}  
  U_{n,\widetilde{a}} (q_1, \dotsc, q_{n-1} ) 
\\
& \times 
  V_{n,b}\Bigl( \Bigl(s-\frac{a_1}{n} \Bigr)- q_1,\dotsc,
    \Bigl( s- \frac{a_1}{n} \Bigr) - q_{n-1},n \Bigl(s-\frac{a_1}{n}\Bigr)+|a| \Bigr)
 \, dq_1 \cdots dq_{n-1}.
\end{align*}
The  proof of part (2) of our lemma then follows from the definition of $ J_{0,n;\widetilde{a},b} $ and from Mellin inversion, and we are done. 
\end{proof}


We now have
\begin{thm} \label{SOGL}
For $ a = (a_1,\dotsc, a_m) $ and $ b=(b_1,\dotsc, b_n) $, we set
\begin{align*}
  L(s; a,b,r) &= \prod_{1 \le j \le m, 1 \le k \le n} \Gamma_{\R}(s+a_j+b_k)\Gamma_{\R}(s+a_j-b_k), \\
  L(s;a, \wedge^2) &= \prod_{1 \le j<k \le m} \Gamma_{\R}(s+a_j+a_k).
\end{align*}Also, let $\ell=m-n$.  Then

\begin{align*}
  &J_{\ell,n;a,b}(s)\\& = \begin{cases}\ds\frac{L(s;a,b,r)}{L(2s;a,\wedge^2)}&\mbox{ if } \ell=0\mbox{ or }1;\\\\\ds\frac{L(s;a,b,r)}{L(2s;a,\wedge^2)} 
  \cdot \frac{1}{4 \pi i} \int_{w} \frac{ \prod_{k=1}^n \Gamma_{\R}(w+b_k) \Gamma_{\R}(w-b_k)}
  { \prod_{j=1}^m \Gamma_{\R}(1-s+w-a_j) \Gamma_{\R} (s+w+a_j) } \,dw&\mbox{ if } \ell=-1.\end{cases}
\end{align*}

\end{thm}


\begin{proof}
Let us first consider the cases $\ell=0$ and $\ell=1$,   using an induction argument.  By Barnes' first lemma, we compute readily that
$$ 
 J_{1,2;a,b}(s) = 
 \frac{ \Gamma_{\R} (s+a_1+b_1) \Gamma_{\R}(s+a_1-b_1) 
        \Gamma_{\R} (s+a_2+b_1) \Gamma_{\R}(s+a_2-b_1) }
   { \Gamma_{\R}(2s+a_1+a_2) }.
$$
Then Proposition \ref{SOGLlem} and the induction hypothesis yield the desired results.

\medskip

Let us now prove our theorem in the case $\ell=-1$.
Our argument here is quite analogous to the proof of Theorem \ref{GLunip}.

For $ w \in \C $, we  compute the integral
\begin{align*}
 J_{-1,n;a,b}(s,w) 
& = 2^n  \int_{(\R_+)^{n-1}} \int_{\R^{n-1}}
  W_{n-1,a}^A(\alpha[y_1,\dotsc, y_{n-1}])  \\
 & \times 
  W_{n,b}^B( u[x_1,\dotsc,x_{n-1}]  \beta[y_1,\dotsc,y_{n-1},y_n ] ) 
\\
 & \times 
  \biggl[ \prod_{j=1}^{n-1} y_j^{j(s-({3n-3)}/{2}+j)} \biggr] 
  y_n^w  
   \prod_{j=1}^{n-1} dx_j  \prod_{j=1}^n \frac{dy_j}{y_j} ,
\end{align*}
as follows.  We use the Iwasawa decomposition of $ u[x_1,\dotsc, x_{n-1}] \beta[y_1,\dotsc,y_n] $ 
(see the begining of the proof of Theorem \ref{GLunip}), together with
the substitutions $ x_j \to (\prod_{p=j}^{n-1}y_p^{-1}) x_j $ for $ 1 \le j \le n-1 $.
We follow this with the substitutions
\begin{align*}
 y_j \to \frac{p_{j+1}}{\sqrt{p_j p_{j+2}}}y_j \ (1 \le j \le n-1), 
\ \ y_n \to \frac{1}{\sqrt{p_n}} y_n, 
\end{align*}
where 
$ p_j = 1 + \sum_{p=1}^{j-1} x_p^2 $ $ (1 \le j \le n) $ and $ p_{n+1} = 1 $.
We thereby arrive at  the following result:
\begin{align*}
 J_{-1,n;a,b}(s,w)
& = 2^n 
   \int_{(\R_+)^{n}} \int_{\R^{n-1}} 
     \exp \biggl\{ -2\pi i \biggl( \sum_{j=1}^{n-2} \frac{x_j x_{j+1}}{ \sqrt{p_jp_{j+2}} } y_j
                  + \frac{x_{n-1}}{ \sqrt{p_{n-1}} } y_{n-1} \biggr) \biggr\}
\\
 & \times \widehat{W}_{n,b}^B( \beta[y_1,\dotsc, y_{n}]) 
\\
 & \times \widehat{W}_{n-1,a}^A 
  \biggl(\alpha\biggl[ \frac{p_2}{\sqrt{p_1p_3}} y_1, \dotsc, 
           \frac{p_{n-1}}{\sqrt{p_{n-2}p_{n}}} y_{n-2}, \frac{p_{n}}{\sqrt{p_{n-1}}} y_{n-1}
  \biggr] \biggr)
\\
 & \times \biggl[ \prod_{j=1}^{n-1} y_j^{js} \biggr] y_{n}^{w+  {n^2}/{2}}
   p_{n}^{ {n(2s-n)}/{4} - ({w+1})/{2}}  
  \prod_{j=1}^{n-1} \frac{dx_j}{\sqrt{p_j}} 
  \prod_{j=1}^{n}  \frac{dy_j}{y_j}.
\end{align*} 
Next we change variables from $ (x_1,\dotsc, x_{n-1}) $ to $ (x_1',\dotsc,x_{n-1}') $, with 
 $ x_j' = x_j \sqrt{ p_n/ (p_j p_{j+1}) } $ for each $j$, and integrate with respect to variables $ x_j' $ 
by way of Proposition \ref{GLrec2}. We find that
\begin{align*}
 J_{-1,n;a,b}(s,w)
& = \frac{2^n}{ \prod_{j=1}^{n-1} \Gamma_{\R}(w-|a|-a_j-n(s-\frac{n}{2})+1)} 
\\
& \times \int_{(\R_+)^{n}} 
   \widehat{W}_{n,b}^B( \beta[y_1,\dotsc, y_{n}]) 
   \widehat{W}_{n,\alpha}^A (\alpha[y_1,\dotsc,y_n])
 \biggl[ \prod_{j=1}^{n} y_j^{j( {n}/{2}+ ({w-|a|})/{n}) } \frac{dy_j}{y_j} \biggr]
\\
& = \frac{1}{ \prod_{j=1}^{n-1} \Gamma_{\R}(w-|a|-a_j-n(s-\frac{n}{2})+1)} 
  \cdot J_{0,n;\breve{a}, b} \Bigl(\frac{n}{2} + \frac{w-|a|}{n} \Bigr),
\end{align*}
where $ \breve{a} = (\breve{a}_1,\dotsc, \breve{a}_n) $ with 
\begin{align*}
 \breve{a}_1 &= -(n-1)(s-\frac{n}{2}) + \frac{n-1}{n}(w-|a|), \\ 
 \breve{a}_j &= a_{j-1} + s - \frac{n}{2} - \frac{1}{n} (w-|a|) \ \ (2 \le j \le n).
\end{align*}
Therefore, the case $\ell=0$ of the present theorem  implies that 
\begin{align*} 
& J_{-1,n;a,b}(s,w)
\\
 & = \frac{ \prod_{1 \le j \le n-1, 1 \le k \le n} \Gamma_{\R}(s+a_j+b_k) \Gamma_{\R}(s+a_j-b_k) }
    { \prod_{1 \le j<k \le n-1} \Gamma_{\R}(2s+a_j+a_k) }
\\
 & \times \frac{ \prod_{k=1}^n \Gamma_{\R}(-(n-1)s+\frac{n^2}{2}+w-|a|+b_k)
      \Gamma_{\R}(-(n-1)s+\frac{n^2}{2}+w-|a|-b_k) }
   {  \prod_{j=1}^{n-1} \Gamma_{\R}(-ns+\frac{n^2}{2}+w-|a|-a_j+1)
     \Gamma_{\R}(-(n-2)s+\frac{n^2}{2} + w-|a|+a_j )  }.
\end{align*}
Thus Mellin inversion and a change of variable yield our assertion, and 
we complete the proof of Theorem \ref{SOGL}.
\end{proof} 

\begin{rem} The local functional equation for $SO_{2n+1}\times GL_{n-1}$, as described in \cite{So2}, amounts, at real places, to the assertion that$$J_{-1,n,a,b}(s)\biggl/\biggl[\frac{L(s;a,b,r)}{L(2s;a,\wedge^2)}\biggr] $$is invariant under $(s,a,b)\to (1-s, -a,-b)$. This invariance is readily apparent from Theorem \ref{SOGL} above.\end{rem}
\renewcommand{\thesection}{\Alph{section}}
\setcounter{section}{1}

\section*{Appendix:  Comparison with  unramified computations}

The well-known Kato-Casselman-Shalika formulas describe  unramified (nonarchimedean) Whittaker functions  in terms of characters of 
finite dimensional representations of the corresponding $L$-groups. The 
unramified computations for $SO_{2n+1} \times GL_m $ can then be achieved by combining these realizations of   unramified Whittaker functions with certain facts about the representation theory of $GL_{n}(\C) $ and $ Sp_n(\C) $, (See \cite[Appendix]{GPR}.)
 
In this appendix, we present an alternative, somewhat roundabout, approach to the unramified calculations for $SO_{2n+1}\times GL_n$ and $SO_{2n+1}\times GL_{n+1}$.  This new approach parallels, in a number of ways, our calculations above in the analogous archimedean cases.  We hope that this suggestion of connections between the unramified and archimedean situations might help enable future computations, particularly in the archimedean context.

The general ideas informing this appendix are as follows.  The $p$-adic analogs of our recursive formulas for Mellin transforms of Whittaker functions can be understood as branching laws
for finite dimensional representation of $GL_n(\C)$ and $Sp_n(\C)$.  Moreover, Pieri's rule, which gives the decomposition into irreducibles of the tensor product of 
an irreducible representation with the symmetric $k$th power of the standard representation, can be considered analogous to our Propositions \ref{lemGL} and \ref{lemSO}.
By way of these observations, we may therefore perform the  unramified computationsin a fashion reminescent of  that employed above  -- namely, by reducing the  $ SO_{2n+1}\times GL_{n+1}$ (respectively,  $SO_{2n+1}\times GL_n$) case to that of
 $ SO_{2n+1}\times GL_n$ (respectively, $SO_{2n-1}\times GL_n)$.

We begin by recalling explicit formulas for unramified Whittaker functions.
Let $F$ be a nonarchimedean local field, $\mathcal{O}$ its ring of integers, and 
$ \mathfrak{p} = (\varpi) $ its maximal ideal. 
We set $ q = |\mathcal{O}/\mathfrak{p}| $. 
We take a non-trivial additive character $ \psi $ of $F$ to be unramified.
Let $ \pi_a^A $ (respectively, $ \pi_b^B  $) be the unramified principal series representation
with Satake paramater $ a = {\rm diag} (a_1,\dotsc,a_n) \in GL_n(\C) $
(respectively, $ b  = {\rm diag}(b _1,\dotsc,b _n,b _n^{-1},\dotsc,b _1^{-1}) \in Sp_n(\C) $).
We sometimes write $a = (a_1,\dotsc,a_n)$ and $b =(b _1,\dotsc,b _n) $.
We then consider the $GL_n(\mathcal{O})$- (respectively, $SO_{2n+1}(\mathcal{O})$)-fixed Whittaker function 
$ W_{n,a}^A \in {\mathcal W}(\pi_a^A,\psi) $ 
(respectively, $W_{n,b }^B \in {\mathcal W}(\pi_b^B ,\psi) $),
normalized such that 
$ W_{n,a}^A (1_n) = W_{n,b }^B (1_{2n+1}) = 1 $.

For $ \lambda = (\lambda_1,\lambda_2,\dotsc, \lambda_n) \in \Z^{n}$, we put
\begin{align*}
\alpha[\lambda] &=\alpha[\lambda_1,\dotsc,\lambda_n] 
  = {\rm diag}({\varpi}^{\lambda_1},\dotsc, {\varpi}^{\lambda_n}) \in GL_n(F), \\
\beta[\lambda] &=\beta[\lambda_1,\dotsc,\lambda_n] 
  = {\rm diag}({\varpi}^{\lambda_1},\dotsc, {\varpi}^{\lambda_n},1,
               {\varpi}^{-\lambda_n},\dotsc,{\varpi}^{-\lambda_1}) \in SO_{2n+1}(F). 
\end{align*} 
We set
\begin{align*}
 W_{n,a}^A(\alpha[\lambda]) & = \delta^{1/2}_A(\alpha[\lambda]) \widehat{W}_{n,a}^A(\alpha[\lambda]),
 \\
 W_{n,b }^B(\beta[\lambda]) & = \delta^{1/2}_B(\beta[\lambda]) \widehat{W}_{n,b }^B(\beta[\lambda]).
\end{align*}
where $\delta_A$ (respectively, $\delta_B$) is the modulus character for the Borel subgroup of 
$ GL_n $ (respectively, $SO_{2n+1}$).
According to the formulas of Shintani and Kato-Casselman-Shalika, we have 
\begin{align*}
 \widehat{W}_{n,a}^A(\alpha[\lambda]) &= 
 \begin{cases} \chi_{\lambda}^A(a) & \mbox{ if } \lambda \in \mathcal{P}_n, \\
               0 & \mbox{ otherwise, } \end{cases} \\
 \widehat{W}_{n,b }^B(\beta[\lambda]) &= 
 \begin{cases} \chi_{\lambda}^C(b ) & \mbox{ if } \lambda \in \mathcal{P}_n^+, \\
               0 & \mbox{ otherwise, } \end{cases} 
\end{align*}             
where $ \chi_{\lambda}^A $ (respectively, $ \chi_{\lambda}^C $) is the character of the irreducible 
finite dimensional representation $ \pi_{\lambda}^A $ (respectively, $\pi_{\lambda}^C) $
of $ GL_{n}(\C) $ (respectively, $Sp_n(\C)$) with highest weight $\lambda $. Here 
\begin{align*}
\mathcal{P}_n &= \{ \lambda = (\lambda_1,\lambda_2, \dotsc,\lambda_n) \in \Z^n
  \mid \lambda_1 \ge \lambda_2 \ge \cdots \ge \lambda_n \},
\\
\mathcal{P}_n^+ &= \{ \lambda =(\lambda_1,\lambda_2, \dotsc,\lambda_n) \in \mathcal{P}_n 
  \mid \lambda_n \ge 0 \}. 
\end{align*}

We define subsets of $ \mathcal{P}_n^+ $ as follows.
For a given $ \lambda \in \mathcal{P}_n^+ $ and a nonnegative integer $r$, we set 
\begin{align*}
 \mathcal{P}_n^+(\lambda) & = \{ \mu=(\mu_1,\dotsc,\mu_n) \in \mathcal{P}_n^+ 
 \mid  \lambda_1 \ge \mu_1 \ge \lambda_2 \ge \cdots \ge \mu_{n-1} \ge \lambda_{n} \ge \mu_n \};
\\
 \mathcal{P}_{n-1}^+(\lambda) & = \{ \mu=(\mu_1,\dotsc,\mu_{n-1}) \in \mathcal{P}_{n-1}^+ 
 \mid  \lambda_1 \ge \mu_1 \ge \lambda_2 \ge \cdots \ge \lambda_{n-1} \ge \mu_{n-1} \ge \lambda_{n} \};
\\
 \overline{\mathcal{P}}_n^+(\lambda) 
  & = \{ \mu \in \mathcal{P}_n^+ \mid \lambda \in \mathcal{P}_n^+(\mu) \} \\
  & = \{ \mu=(\mu_1,\dotsc,\mu_n) \in \mathcal{P}_n^+ 
 \mid  \mu_1 \ge \lambda_1 \ge \mu_2 \ge \cdots \ge \lambda_{n-1} \ge \mu_{n} \ge \lambda_n \}; 
\\
 \mathcal{P}_{n,r}^+(\lambda) & = \{ \mu \in \mathcal{P}_n^+(\lambda) \mid |\lambda|-|\mu| = r \};
\\\
 \overline{\mathcal{P}}_{n,r}^+(\lambda) &
  = \{ \mu \in \overline{\mathcal{P}}_n^+(\lambda) \mid \lambda \in \mathcal{P}_{n,r}^+(\mu) \} 
  = \{ \mu \in \overline{\mathcal{P}}_n^+(\lambda) \mid |\mu|-|\lambda| = r \}.
\end{align*}
Here we write $ |\lambda|=\sum_{i=1}^n \lambda_i $ for $ \lambda \in \mathcal{P}_n^+ $.

\medskip 

We first state the nonarchimedean analogs of the recursive formulas in 
Propositions \ref{MBGL} and \ref{MBSO}. 
Let $ \lambda \in \mathcal{P}_n^+ $.  
We have the following:  
\begin{align} \label{pGL}
 \widehat{W}_{n,a}^A(\alpha[\lambda]) &
 = \sum_{\mu \in \mathcal{P}_{n-1}^+(\lambda)} 
   a_n^{|\lambda|-|\mu|} \widehat{W}_{n-1,\widetilde{a}}^A (\alpha[\mu])
\end{align}
and 
\begin{align} \label{pSO}
 \widehat{W}_{n,b }^B(\beta[\lambda]) &
 = \sum_{\nu \in \mathcal{P}_n^+(\lambda)} \sum_{\mu \in \mathcal{P}_{n-1}^+(\nu)} 
   b _n^{2|\nu|-|\lambda|-|\mu|} \widehat{W}_{n-1,\widetilde{b }}^B (\beta[\mu]), 
\end{align}
where $ \widetilde{a} = (a_1,\dotsc, a_{n-1}) $ and 
$ \widetilde{b } = (b _1,\dotsc, b _{n-1}) $.
These formulas can be interpreted as the branching laws
\begin{align*}
 \pi_{\lambda}^A|_{GL_{n-1}(\C)} & = \bigoplus_{\mu \in \mathcal{P}_{n-1}^+(\lambda)} \pi_{\mu}^A, \\
 \pi_{\lambda}^C|_{Sp_{n-1}(\C)} & = \bigoplus_{\mu \in \mathcal{P}_{n-1}^+(\lambda)} m(\lambda,\mu) \pi_{\mu}^C,
\end{align*}
where the mulitiplicity $ m(\lambda,\mu) $ is the cardinality of 
the set $ \{ \nu \in \mathcal{P}_n^+(\lambda) \mid  \mu \in \mathcal{P}_{n-1}^+(\nu) \}. $

\medskip 

To discuss a nonarchimedean analog of the formulas in Propositions \ref{lemGL} and \ref{lemSO}, 
let us recall Pieri's rule. This is a special case of the Littlewood-Richardson rule  (\cite[\S A.1]{FuHa}), which describes 
how the tensor product $ \pi_{\lambda}^A \otimes \pi_{\kappa}^A $ decomposes:  
\begin{align*}
 \pi_{\lambda}^A \otimes \pi_{\kappa}^A \cong \oplus_{\mu} M_{\lambda \kappa}^{\mu} \pi_{\mu}^A. 
\end{align*}
If we write $ (k) := (k,0,\dotsc,0) \in \mathcal{P}_n^+ $, 
then Pieri's rule (\cite[Proposition 15.25]{FuHa}) asserts 
\begin{align} \label{Pieri}
 M_{\lambda (k)}^{\mu} = 
 \begin{cases} 1 & \mbox{ if } \mu \in \overline{\mathcal{P}}_{n,k}^+(\lambda), \\
               0 & \mbox{ otherwise. } \end{cases}
\end{align}  
We denote by $ N_{\lambda \kappa}^{\mu} $ the Littlewood-Richardson coefficient for $ Sp_n(\C) $:
\begin{align*}
 \pi_{\lambda}^C \otimes \pi_{\kappa}^C \cong \oplus_{\mu} N_{\lambda \kappa}^{\mu} \pi_{\mu}^C. 
\end{align*}
It is known that $ N_{\lambda \kappa}^{\mu} $ may be expressed in terms of 
$ M_{\lambda \kappa}^{\mu} $ as follows:
$$
   N_{\lambda \kappa}^{\mu} = \sum_{\nu, \sigma, \tau} 
   M_{\sigma \nu}^{\lambda} M_{\tau \nu}^{\mu} M_{\sigma \tau}^{\kappa}
$$
(\cite[\S 25.3]{FuHa}).
From the Littlewood-Richardson rule, we find that 
\begin{align*} 
 M_{\sigma \tau}^{(k)} 
= \begin{cases}   
  1  & \mbox{ if }  \sigma = (r) \mbox{ and } \tau = (k-r) \mbox{ for\ some } 0 \le r \le k, \\   
  0  & \mbox{ otherwise. }
 \end{cases}  
\end{align*}
Then Pieri's rule (\ref{Pieri}) implies a Pieri-type rule for $Sp_n(\C)$:
\begin{equation} \label{PieriSp}
\begin{split}
  N_{\lambda (k)}^{\mu} & = \sum_{0 \le r \le k} 
   \sum_{\nu} M_{(r) \nu}^{\lambda} M_{(k-r) \nu}^{\mu} \\
  & = \sum_{0 \le r \le k} \sum_{ \nu \in \mathcal{P}_{n,r}^{+}(\lambda) } 
    \sum_{ \mu \in \overline{\mathcal{P}}_{n,k-r}^+(\nu)} 1 \\
  & = \sharp \{ \nu \in \mathcal{P}_n^+(\lambda) \cap \mathcal{P}_n^+(\mu) 
      \mid |\lambda|+|\mu|-2|\nu| = k \}.
\end{split}
\end{equation}
Let $ h_k $ denote the $k$th complete symmetric polynomial in $n$ variables:
\begin{align*}
 h_{k}(x_1,\dotsc,x_n) 
  = \sum_{1 \le i_1 \le \cdots \le i_k \le n} x_{i_1} \cdots x_{i_k}. 
 \end{align*}
If we set
\begin{align*}
 h_k^A(a) &= h_r(a_1,\dotsc,a_n), \\
 h_k^B(b ) &= h_r(b _1,\dotsc,b _n,b _1^{-1},\dotsc,b _n^{-1}),
\end{align*}
then $ h_k^A $ and $ h_k^B $ give the characters of representations with 
highest weight $ (k,0,\dotsc,0) $ of $GL_n(\C) $ and $Sp_n(\C) $, respectively.
Thus (\ref{Pieri}) and (\ref{PieriSp}) imply that  
\begin{align*}
 \sum_{\mu \in \overline{\mathcal{P}}_{n,k}^+(\lambda) } \widehat{W}_{n,a}^A(\alpha[\mu]) 
  = h_k^A(a) \widehat{W}_{n,a}^A(\alpha[\lambda])
\end{align*}
and
\begin{align*}
 \sum_{0 \le r \le k} \sum_{\nu \in \mathcal{P}_{n,k-r}^+(\lambda)} 
 \sum_{ \mu \in \overline{\mathcal{P}}_{n,r}^+(\nu) }  \widehat{W}_{n,b }^B(\beta[\mu]) 
  = h_k^B(b ) \widehat{W}_{n,b }^B(\beta[\lambda])
\end{align*}
for $ \lambda \in \mathcal{P}_n^+ $. In light of the formulas
\begin{align*}
 \sum_{k=0}^{\infty} h_k^A(a) t^k&= \prod_{j=1}^n (1-a_j t)^{-1},\\ 
 \sum_{k=0}^{\infty} h_k^B(b ) t^k & = \prod_{j=1}^n \{ (1-b _j t)(1-b _j^{-1} t) \}^{-1},
\end{align*}
we arrive at our nonarchimedean analogs of Propositions \ref{lemGL} and \ref{lemSO}:
\begin{align} \label{pGLlem}
 \sum_{\mu \in \overline{\mathcal{P}}_n^+(\lambda)} 
 t^{|\mu|-|\lambda|} \widehat{W}_{n,a}^A(\alpha[\mu]) = 
 \Bigl[ \prod_{j=1}^n (1-a_j t)^{-1} \Bigr] \widehat{W}_{n,a}^A(\alpha[\lambda]),
\end{align}
\begin{align} \label{pSOlem}
 \sum_{\nu \in \mathcal{P}_n^+(\lambda)} \sum_{\mu \in \overline{\mathcal{P}}_n^+(\nu)} 
 t^{|\lambda|+|\mu|-2|\nu|} \widehat{W}_{n,b }^B(\beta[\mu]) = 
   \Bigl[ \prod_{j=1}^n \{ (1-b _j t)(1-b _j^{-1}t) \}^{-1} \Bigr]
   \widehat{W}_{n,b }^B(\beta[\lambda]).
\end{align}

\medskip 

Using the  above formulas, we now compute the unramified zeta integrals in a manner similar to that employed above, in 
the archimedean situation.
The unramified zeta integrals for $SO_{2n+1}\times GL_n$ and $SO_{2n+1}\times GL_{n+1}$ 
are given by
\begin{align*}
 J_{0,n;a,b }(s) = \sum_{\lambda \in \mathcal{P}_n^+} 
  \widehat{W}_{n,a}^A(\alpha[\lambda]) \widehat{W}_{n,b }^B(\beta[\lambda]) q^{-|\lambda|s} 
\end{align*}
and 
\begin{align*}
 J_{1,n;a,b }(s) = \sum_{\lambda \in \mathcal{P}_n^+}
  \widehat{W}_{n+1,a}^A(\alpha[\lambda,0]) \widehat{W}_{n,b }^B(\beta[\lambda])q^{-|\lambda|s},
\end{align*}
respectively. Here we have written  $\alpha[\lambda,0] =\alpha[\lambda_1,\dotsc,\lambda_n,0].$ 

We can relate $J_{0,n;a,b }(s) $ (respectively, $J_{1,n;a,b }(s) $) to 
$ J_{1,n-1;a,\widetilde{b }}(s) $ (respectively, $J_{0,n;a,b }(s) $) as follows.
From (\ref{pSO}), we have
\begin{align*}
 J_{0,n;a,b }(s) & = \sum_{\lambda \in \mathcal{P}_n^+} 
 \sum_{\nu \in \mathcal{P}_n^+(\lambda)} \sum_{\mu \in \mathcal{P}_{n-1}^+(\nu)}
 \widehat{W}_{n,a}^A(\alpha[\lambda]) \widehat{W}_{n-1,\widetilde{b }}^B(\beta[\mu])
 b _n^{2|\nu|-|\lambda|-|\mu|} q^{-|\lambda|s}.
\end{align*}
We use (\ref{pGLlem}) with $ t= b _n^{-1} q^{-s} $ and $ t=b _n q^{-s} $ successively
to sum over $ \lambda $ and $ \nu $:
\begin{align*}
& J_{0,n;a,b }(s) \\
& = \Bigl[ \prod_{j=1}^n (1-a_j b _n^{-1} q^{-s})^{-1} \Bigr]  
 \sum_{\nu \in \mathcal{P}_n^+} \sum_{\mu \in \mathcal{P}_{n-1}^+(\nu)}
 \widehat{W}_{n,a}^A(\alpha[\nu]) \widehat{W}_{n-1,\widetilde{b }}^B(\beta[\mu])
 b _n^{|\nu|-|\mu|} q^{-|\nu|s} \\
&= \Bigl[ \prod_{j=1}^n \{ (1-a_j b _n q^{-s})(1-a_j b _n^{-1} q^{-s}) \}^{-1} \Bigr]
   \sum_{\mu \in \mathcal{P}_{n-1}^+} 
   \widehat{W}_{n,a}^A(\alpha[\mu,0]) \widehat{W}_{n-1,\widetilde{b }}^B(\beta[\mu]) q^{-|\mu|s}
\\
&= \Bigl[ \prod_{j=1}^n \{(1-a_j b _n q^{-s})(1-a_j b _n^{-1} q^{-s})\}^{-1} \Bigr]
    J_{1,n-1;a,\widetilde{b }}(s).
\end{align*}

Now, let us consider $J_{ 1,n;a,b }(s) $. We observe that  (\ref{pGL}) implies
\begin{align*}
 J_{1,n;a,b }(s)
& = \sum_{\lambda \in \mathcal{P}_n^+} \sum_{\mu \in \mathcal{P}_n^+(\lambda)} 
    \widehat{W}_{n,\widetilde{a}}^A(\alpha[\mu]) \widehat{W}_{n,b }^B(\beta[\lambda]) 
    a_{n+1}^{|\lambda|-|\mu|} q^{-|\lambda|s}. 
\end{align*}
Apply (\ref{pGLlem}) with $ t=a_{n+1} q^{-2s} $ to $ \widehat{W}_{n,\widetilde{a}}^A(\alpha[\mu]) $, 
to get 
\begin{align*}
& \Bigr[ \prod_{j=1}^{n} (1-a_j a_{n+1}q^{-2s})^{-1} \Bigr] J_{1,n;a,b }(s) \\
& = \sum_{\lambda \in \mathcal{P}_n^+} \sum_{\mu \in \mathcal{P}_n^+(\lambda)} 
    \sum_{\nu \in \overline{\mathcal{P}}_n^+(\mu)} 
    \widehat{W}_{n,\widetilde{a}}^A(\alpha[\nu]) 
    \widehat{W}_{n,b }^B(\beta[\lambda]) 
    a_{n+1}^{|\lambda|+|\nu|-2|\mu|} q^{-(|\lambda|+2|\nu|-2|\mu|) s}
\\
& = \sum_{ \nu \in \mathcal{P}_n^+ } \widehat{W}_{n,\widetilde{a}}^A(\alpha[\nu])
    q^{-|\nu|s}
   \sum_{\mu \in \mathcal{P}_n^+(\nu)}
  \sum_{\lambda \in \overline{\mathcal{P}}_n^+(\mu)} \widehat{W}_{n,b }^B(\beta[\lambda])
  (a_{n+1} q^{-s})^{|\nu|+|\lambda|-2|\mu|}.
\end{align*}
By (\ref{pSOlem}) with $ t = a_{n+1} q^{-s} $, we can sum over 
$ \lambda $ and $ \mu $, to get 
\begin{align*}
 J_{1,n;a,b }(s)
= \frac{ \prod_{j=1}^{n} \{(1-a_{n+1}b _j q^{-s})(1-a_{n+1}b _j^{-1} q^{-s})\}^{-1} }
  { \prod_{j=1}^{n} (1-a_j a_{n+1}q^{-2s})^{-1} } J_{0,n;\widetilde{a},b }(s) ,
\end{align*} and we are done.

\end{document}